\documentclass[journal]{IEEEtran}

\usepackage{cite}
\usepackage{lastpage}										
\usepackage{indentfirst,dsfont}								
\usepackage{color,parskip}											
\usepackage{graphicx}										

\usepackage{mathtools,balance}
\usepackage[section]{placeins}
\usepackage{amssymb}
\usepackage{indentfirst}
\usepackage{psfrag}
\usepackage{tabularx}
\usepackage{algorithm,bbm}
\usepackage{algorithmic,balance}
\usepackage{amsthm,import,pdfpages,transparent,xifthen}
\usepackage{subfigure} 
\usepackage{wasysym,tcolorbox,subfig}
\usepackage{graphicx,tabularx,epsfig,epstopdf}

\newtheorem{assump}{\textbf{Assumption}}

\newtheorem{prop}{\textbf{Proposition}}
\newtheorem{theo}{\textbf{Theorem}}

\newtheorem{remark}{Remark}

\newtheorem{defi}{Definition}

\DeclareMathOperator*{\minimise}{minimize}
\DeclareMathOperator*{\argmin}{argmin}
\DeclareMathOperator*{\argmax}{argmax}

\graphicspath{{Figures/}}

\ifCLASSINFOpdf
\else
\fi
\begin{document}

\title{On the exact feasibility of convex scenario programs with discarded constraints}
\author{Licio Romao, Antonis~Papachristodoulou, and Kostas Margellos 
	\thanks{L. Romao is supported by the Coordination for the Improvement of Higher Education Personnel (CAPES) - Brazil. The work of A. Papachristodoulou and K. Margellos has been supported by EPSRC UK under grants $ \mathrm{EP}/\mathrm{M}002454/1 $ and $ \mathrm{EP}/\mathrm{P}03277\mathrm{X}/1  $, respectively.}
	\thanks{The authors are with the Department of Engineering Science, University of Oxford, Parks Road, Oxford OX1 3PJ, UK. {\tt\small \{licio.romao,  antonis, kostas.margellos\}@eng.ox.ac.uk.}}
}

\markboth{Journal of \LaTeX\ Class Files,~Vol.~14, No.~8, August~2015}%
{Shell \MakeLowercase{\textit{et al.}}: Bare Demo of IEEEtran.cls for IEEE Journals}

\maketitle

\begin{abstract}
We revisit the so-called sampling and discarding approach used to quantify the probability of constraint violation of a solution to convex scenario programs when some of the original samples are allowed to be discarded. Motivated by two scenario programs that possess analytic solutions and the fact that the existing bound for scenario programs with discarded constraints is not tight, we analyze a removal scheme that consists of a cascade of optimization problems, where at each step we remove a superset of the active constraints. By relying on results from compression learning theory, we show that such a removal scheme leads to less conservative bounds for the probability of constraint violation than the existing ones. We also show that the proposed bound is tight by characterizing a class of optimization problems that achieves the given upper bound. The performance improvement of the proposed methodology is illustrated by an example that involves a resource sharing linear program.
\end{abstract}

\begin{IEEEkeywords}
Scenario approach, randomized algorithms, chance-constrained optimization, probabilistic methods.
\end{IEEEkeywords}

\IEEEpeerreviewmaketitle

\section{Introduction}
%
%
%
%

\IEEEPARstart{U}ncertain optimization programs capture a wide class of engineering applications. Tractability of this class of optimization problems is an active area of research~\cite{Vid:98,ATC:09,MGL:14,CGC:15,Cal:17,Ram:18,CGR:18,CG:18}. In the last decades, several approaches have been developed to cope with uncertainty in an optimization context. Among those, robust optimization~\cite{BS:05,BBC:11,BN:08,BGN:09} has been successfully applied to several control problems~\cite{dOBG:99,FDT:03,GdOB:03,CDT:11,TCD:12,RdOPO:16}. It consists of making certain assumptions, often arbitrary, on the geometry of the uncertainty set (ellipsoidal, polytopic, etc.) and then optimizing over the worst case performance within this set. Another approach is chance-constrained optimization~\cite{Pre:70,NS:06,PAS:09} that relies on imposing constraints that only need to be satisfied with given probability. However, optimization problems with chance-constraints are hard to solve in general, without imposing any assumption on the underlying distribution of the uncertainty (e.g., Gaussian).

An alternative to robust and chance-constrained optimization involves data driven algorithms. Within  this context, this paper lies in the realm of the scenario approach theory~\cite{CC:05,CC:06,CG:08,CG:11,Cal:10,MPL:15,CG:18,CG:18b,GC:19}: a randomized technique which involves generating a finite number of scenarios and enforcing a different constraint for each of them. Under convexity, the optimal solution to such a scenario program is shown to be feasible (with certain probability) to the associated chance-constrained program. One of the fundamental developments in the scenario approach literature is to provide a distribution-free bound on the probability of constraint violation that holds for all convex problems~\cite{CG:08}. Moreover, this bound is tight in the sense that it is achieved by the so-called class of fully-supported optimization problems, however, it might be conservative for more general problem classes.

To alleviate this conservatism and trade feasibility to performance, the so-called sampling and discarding~\cite{CG:11} (see also \cite{CG:18b}) was introduced; a similar result known as scenario approach with constraint removal was also developed in \cite{Cal:10}. These allow removing some of the extracted scenarios and enforcing the constraints only on the remaining ones, thus improving the performance in terms of optimality of the resulting solution. As opposed to the original bound in \cite{CG:08}, however, the bound on the probability of constraint violation in~\cite{CG:11,Cal:10} is not tight. 

Similarly to the motivation of \cite{CG:11}, our main goal is to improve performance and decrease the conservatism of the solution obtained by means of the scenario approach theory. We capitalize on the fact that the bound of the sampling-and-discarding scheme is not tight, to provide a less conservative and tight bound on the probability of constraint violation for convex scenario programs with discarded constraints. To this end, we develop a novel analysis approach to study a removal procedure that consists of solving a cascade of scenario programs and removing, at each stage, scenarios in an integer multiple of the dimension of the decision variables. Our theoretical findings bear important consequences in the application of the scenario theory to control problems \cite{CC:06,CL:13,MMB:15,VTKQ:20,SOZZ:20,KP:20,DA:20,CG:20,CJJKT:20,JP:20,SYN:20,SAP:20}, as we may be able to achieve better performance while guaranteeing the same level of constraint violation and confidence. The proposed bound on the probability of constraint violation is similar in terms of complexity to the one of \cite{CG:11,Cal:10}; it is also distribution-free and holds, under a non-degeneracy assumption (to be formally defined in the sequel), for all convex problems. We also show that the resulting bound is tight, and characterize the class of scenario programs for which this is the case. As such, our results extend the ones of \cite{CG:11,Cal:10} and cannot be further improved. To summarize, our main contributions are: 

\begin{itemize}
	\item Proposing and analyzing a removal scheme that possesses tighter guarantees than \cite{CG:11,Cal:10} on the probability of constraint violation for scenario programs with discarded constraints (Theorems \ref{theo:main_result_fully_supp} and \ref{theo:main_result_non_deg}).
	\item Proving tightness of the resulting bound by characterizing the class of scenario programs that satisfies our bound with equality (Theorem \ref{theo:main_result_tight}).
    \item Computing analytically the solution of two scenario programs where the bound is tight (Section \ref{sec:analytic_sol}).
	\item Relaxing an assumption present in \cite{CG:11} that requires the removed scenarios to be violated by the final solution.
	\item Developing a novel proof line. Our analysis departs from the one of~\cite{CG:11}, and is based on probably approximately correct (PAC) learning concepts that use the notion of compression~\cite{MPL:15,Vid:02,FW:95}.
\end{itemize}

It is important to highlight that our analysis holds for a particular discarding scheme, which requires removing scenarios in batches, preventing us to remove them one by one. Extension to this direction is outside the scope of the current paper. Moreover, all our results are \emph{a priori}; possibly less conservative but \emph{a posteriori} results are available~\cite{CG:18,CGR:18,GC:19}, however, follow a different conceptual and analysis line from the one adopted in this paper.

The paper is organized as follows: Section~\ref{sec:intro_main_concepts} reviews some background results on the scenario approach with discarded constraints and certain learning theoretic concepts. Section \ref{sec:analytic_sol} motivates the main results of the paper by means of two scenario programs that possess analytic solutions. Section~\ref{sec:discard_scheme} introduces the proposed scenario discarding scheme and states the main results of the paper. Proofs are provided in Section~\ref{sec:proof_main_results}. Section \ref{sec:tight} characterizes the class of optimization programs for which the proposed result is tight and Section \ref{sec:exam} illustrates the theoretical results by means of a numerical example. Finally, Section \ref{sec:concl} concludes the paper and provides some directions for future work.

\section{Scenario optimization with discarded scenarios}
\label{sec:intro_main_concepts}

\subsection{Sampling and discarding}
 \label{sec:prob_state}

Let $ \Delta $ be the space where an uncertainty vector takes values from and denote by $ (\Delta,\mathcal{F},\mathbb{P}) $ the associated probability space, where $ \mathcal{F} $ is a $ \sigma$-algebra and $ \mathbb{P}:\mathcal{F} \rightarrow [0,1]$ is a probability measure on $ \Delta $ (see~\cite{Sal:16} for more details). Fix any $m \in \mathbb{N}$ and let $ S = \{ \delta_1, \delta_2, \ldots, \delta_m\} $ be independent and identically distributed (i.i.d.) samples from $ \mathbb{P}$. 
Note that $ (\delta_1,\ldots,\delta_m) \in \Delta^m $; a natural probability space associated with $ \Delta^m $ is $ (\Delta^m, \otimes_{i = 1 }^m \mathcal{F}, \mathbb{P}^m) $, where $ \otimes_{i = 1}^m \mathcal{F} $ is the smallest $ \sigma$-algebra containing the cone sets $ \prod_{i = 1}^m \mathcal{F}$.
Our analysis is based on a data-driven interpretation, where $\mathbb{P}$ is considered to be fixed, but possibly unknown, and the only information about uncertainty is a collection of i.i.d. scenarios $S$.

We consider convex optimization programs affected by uncertainty $\delta$, and represent uncertainty by means of scenarios. This gives rise to the so-called convex scenario programs, where constraints are enforced only on the scenarios in $ S $~\cite{CC:05,CG:08,Cal:10}. We are particularly interested in the case where some of the scenarios are removed, in view of improving the performance of the obtained solution.
This is known as \emph{sampling and discarding} in the terminology of \cite{CG:11} (also known as scenario approach with constraint removal in \cite{Cal:10}).

To this end, for any set $R \subset S$, with $| R | = r < m$, consider the following problem
\begin{equation}
\begin{aligned}
\minimise_{x \in X}  &\quad c^\top x \\ 
\mathrm{subject~to}  &\quad g(x,\delta) \leq 0, \text{ for all } \delta \in S\setminus R,
\end{aligned}
\label{eq:P_k}
\end{equation}
where $ x \in \mathbb{R}^d $, $ X $ is a closed and convex set of $ \mathbb{R}^d $, and function $ g:\mathbb{R}^d \times \Delta \rightarrow \mathbb{R} $ is convex in $ x $ for all $ \delta \in \Delta$. The subset $ R $ contains scenarios that have been removed by means of a procedure that uses $ S $ as input; hence, strictly speaking $ R $ depends on the scenarios $ S $ but this dependency is omitted for simplicity. If $R = \emptyset$, then one recovers the standard scenario approach \cite{CC:06,CG:08}. Moreover, the objective function is taken to be affine without loss of generality; in case of an arbitrary convex objective function, an epigraphic reformulation would render the problem in the form of \eqref{eq:P_k}. Note that only convex scenarios programs will be considered, as in~\cite{CG:11,Cal:10}. 
\begin{assump}[Feasibility, Uniqueness]
	For any $ S \subset \Delta^m, R \subset S $, the optimal solution of~\eqref{eq:P_k} exists and is unique.
	\label{assump:Feas_Uniq}
\end{assump}
In case of multiple solutions a convex tie-break rule could be selected to single-out a particular one, thus relaxing the uniqueness requirement of Assumption \ref{assump:Feas_Uniq}.

Denote by $x^\star(S)$ the (unique under Assumption \ref{assump:Feas_Uniq}) minimizer of \eqref{eq:P_k}. Note that we introduce $S$ as argument since the optimal solution of~\eqref{eq:P_k} is a random variable that depends on all extracted scenarios, i.e., it is a random variable that takes values on the space $ \Delta^m $. The following result from \cite{CG:11} characterizes the probability that $x^\star(S)$ violates the constraints for a new realization of $\delta$ exceeds a given level $\epsilon \in (0,1)$.

\begin{theo}[Theorem~2.1,~\cite{CG:11}, or Theorem~4.1,~\cite{Cal:10}]
Consider Assumption \ref{assump:Feas_Uniq}, and fix $ \epsilon \in (0,1) $. Let $m > d+r$ and denote by $x^\star(S)$ the optimal solution of \eqref{eq:P_k}. If with $\mathbb{P}^m-$probability one all removed scenarios are violated by the resulting solution $x^\star(S)$, i.e., $g(x^\star(S),\delta)>0$ for all $\delta \in R$, with $\mathbb{P}^m$-probability one, then
\begin{align}
\mathbb{P}^m \bigg\{ &(\delta_1,\ldots,\delta_{m}) \in \Delta^m: \mathbb{P} \big\{ \delta \in \Delta: g(x^\star(S),\delta) > 0 \big\}  > \epsilon  \bigg\} \nonumber  \\ 
& \hspace{1.2cm} \leq {r+d-1 \choose r} \sum_{i = 0}^{r+d-1} {m \choose i}\epsilon^i (1-\epsilon)^{m-i}.\label{eq:thm_sampl-disc}
\end{align} 
\label{theo:sampl-disc}
\end{theo}

Theorem \ref{theo:sampl-disc} represents an important generalization of the scenario approach theory, as it allows the decision maker to trade feasibility to performance. Indeed, observe that the feasible set of \eqref{eq:P_k} is enlarged when $ R $ is non-empty (i.e., when scenarios are discarded), thus leading to a cost improvement with respect to the case where $ R $ is the empty set. This fact and the bound of Theorem \ref{theo:sampl-disc} enable the decision maker to improve cost, while controlling the probability of constraint violation.

It should be also noted that Theorem \ref{theo:sampl-disc} does not allow for an arbitrary discarding scheme; it rather requires that, with $\mathbb{P}^m$-probability one, all discarded scenarios are violated by the resulting solution $x^\star(S)$. This is instrumental in the proof of Theorem~\ref{theo:sampl-disc}, as shown in \cite{CG:11}. 


Besides, if $ r = 0 $ the bound in Theorem \ref{theo:sampl-disc} is known to hold with equality for a class of scenario programs called fully-supported programs (see Section \ref{sec:intro_main_concepts} or \cite{CG:08,CG:11} for more details). When scenarios are discarded, i.e., when $ r \neq 0 $, it is elusive how to construct a removable scheme that allows for a tight bound. It is shown in Section 4.2 of \cite{CG:11} that 

\small 
\begin{align}
\sup_{\mathcal{P},\mathcal{R}} \mathbb{P}^m \bigg\{ (\delta_1,\ldots,\delta_{m})& \in \Delta^m:~  \mathbb{P} \big\{ \delta \in \Delta: g(x^\star(S),\delta) > 0 \big\}  > \epsilon  \bigg\} \nonumber \\ &\geq \sum_{i = 0}^{r+d-1} { m \choose i } \epsilon^i (1-\epsilon)^{m - i},
	\label{eq:CG_11_lower_bound}
\end{align}
\normalsize

\noindent where $\mathcal{P}$ represents the class of optimization problems in the form of \eqref{eq:P_k} that are parameterized by the set $X$, the objective function's cost vector $c$, the constraint function $g$, and (implicitly through the samples) the probability measure $\mathbb{P}$. The set $\mathcal{R}$ represents the collection of scenario removal schemes that return a solution $x^\star(S)$ that violates all the discarded scenarios. If the supremum is achieved, then \eqref{eq:CG_11_lower_bound} implies that there exists a problem in $\mathcal{P}$ and a removal scheme in $\mathcal{R}$ such that the right-hand side of \eqref{eq:CG_11_lower_bound} constitutes a lower bound for $\mathbb{P}^m \{(\delta_1,\ldots,\delta_m) \in \Delta^m:~ \mathbb{P}\{\delta \in \Delta:~g(x^\star(S),\delta)> 0 \} > \epsilon\}$. In particular, in the proof of \eqref{eq:CG_11_lower_bound} (see Section 5.2 in \cite{CG:11}), it is shown that this lower bound is admitted if the underlying problem is fully-supported (see Definition \ref{def:fully_supp} in the sequel) and the removal scheme, among the minimizers that violate all discarded scenarios, returns the one with the highest probability of constraint violation. However, the latter is not implementable, as it would require knowledge of the underlying probability distribution $\mathbb{P}$ which might be unknown. Even if this was known, computing the probability of constraint violation would require the computation of a multi-dimensional integral which is in general difficult. As such, the result in \eqref{eq:CG_11_lower_bound} is an existential statement; in fact it is not shown whether the lower bound is achievable in the sense that \eqref{eq:CG_11_lower_bound} would hold with equality.

This is in contrast with our main result in Theorem 1 that shows that the right-hand side in (3) is in fact an upper-bound for the confidence with which the probability of constraint violation exceeds $\epsilon$. Moreover, our discarding mechanism is constructive and distribution-free, in the sense that it does not require the knowledge of $\mathbb{P}$ for the computation of the resulting solution that enjoys these properties. We also show that such upper bound is tight (see Theorem \ref{theo:main_result_tight}). To achieve this, in Section \ref{sec:discard_scheme} we introduce an alternative discarding strategy composed by a cascade of optimization problems that, roughly speaking, removes a set of cardinality $d$ containing the active constraints of each stage. As a byproduct of our analysis, we also relax the assumption of Theorem~\ref{theo:sampl-disc} that requires all the removed scenarios to be violated by the final solution.

\subsection{Learning theoretic concepts} \label{sec:learning}
The following definition is crucial for the results in this paper. 

\begin{defi}[Compression set]
	Fix $ m \in \mathbb{N} $, and consider $ S \subset \Delta $ with $|S| = m$. Let $\zeta < m$, and $ C \subset S $ with cardinality $|C| = \zeta$. Consider a mapping $ \mathcal{A} : \Delta^{m} \rightarrow 2^\Delta$. If with $\mathbb{P}^m$- probability one 
	\[ 
	\delta \in \mathcal{A} \left( C\right),~ \text{ for all } \delta \in S,
	\]
	then $ C $ is called a compression set of cardinality $\zeta$ for $ \mathcal{A} $.
	\label{def:compression_set}
\end{defi}
In other words, a compression set $C$ is a subset of the samples $ S $ such that $ \mathcal{A}(C)$, i.e., the set generated using only $\zeta$ of the samples, contains all samples in $ S $, even the ones that were not included in $C$. In statistical learning theory this property is known as consistency of $\mathcal{A}(C)$ with respect to the samples \cite{Vid:02,MPL:15}. 
The main focus within a probably approximately correct (PAC) learning framework (see \cite{Vid:02} and references therein) is to quantify the probability that $\mathcal{A}(C)$ differs from $\Delta$. Since $\mathcal{A}(C)$ depends on the scenarios in $S$ (as $C$ is a selection among all scenarios), this probability is itself a random variable defined on the product probability space $ \Delta^m $.

To address this question we will use tools from PAC learnability related to compression learning. To this end, we adapt the main concepts and result of \cite{MPL:15} to the notation of our paper. 
\begin{theo}[Theorem~3,~\cite{MPL:15}]
	Fix $ \epsilon \in (0,1) $ and $\zeta < m$. If with $\mathbb{P}^m$- probability one there exists a unique compression set $C$ of cardinality $ \zeta $, then
	\begin{align}
		\mathbb{P}^m \bigg\{ (\delta_1,\ldots,\delta_{m}) \in \Delta^m&: \mathbb{P} \big\{ \delta \in \Delta: \delta \notin \mathcal{A}(C) \big\}  > \epsilon  \bigg\} \nonumber  \\ 
		& = \sum_{i = 0}^{\zeta-1} {m \choose i}\epsilon^i (1-\epsilon)^{m-i}.\label{eq:thm_compression}
	\end{align} 
	\label{theo:Kostas_paper}
\end{theo}
For a fixed $ \epsilon \in (0,1) $, observe that the right-hand side of~\eqref{eq:thm_compression} goes to zero as $ m $ tends to infinity. This is a desirable property, as it indicates that $ \Delta $ can be asymptotically approximated by $ \mathcal{A}(C) $. Moreover, for a fixed $ m \in \mathbb{N} $, the result of Theorem~\ref{theo:Kostas_paper} provides a non-asymptotic result, quantifying the measure of the set $ \Delta \setminus\mathcal{A}(C)  $. A mapping with these properties is called PAC within the learning literature. Theorem~\ref{theo:Kostas_paper} states that if a mapping possesses a unique compression set, then it is at least $(1-\epsilon)$-accurate as an approximation of $\Delta$ (approximately correct), with confidence (probably) equal to $1-\sum_{i = 0}^{\zeta-1} {m \choose i}\epsilon^i (1-\epsilon)^{m-i}$.

\section{Motivating example: two scenario programs with analytic solutions}
	\label{sec:analytic_sol}

	\subsection{One-dimensional example}
	Suppose that $m$ i.i.d. samples, $ S = \{ \delta_1, \ldots, \delta_m \} $, are drawn from a uniform distribution on the interval $ [0,1] $. Consider the following scenario program that is in the form of \eqref{eq:P_k}.
	\begin{align}
		\minimise_{x \in [0,1]} & \quad x \nonumber \\
		\mathrm{subject~to} & \quad x \geq \delta, \quad \delta \in S \setminus R
		\label{eq:Sampled_toy_2}.
	\end{align}
	Under the choice of a uniform distribution, the optimal solution of \eqref{eq:Sampled_toy_2} is unique with $\mathbb{P}^m$-probability one. Let $ r < m $ be the number of discarded scenarios and consider a (natural) removal scheme that discards scenarios one by one by means of a cascade of scenario programs where at each stage the scenario corresponding to the active constraint is removed from the set $ S $. For instance, the first discarded scenario, which can be explicitly computed as $ \delta^{(1)} =  \argmax_{\delta \in S} \delta $, corresponds to the active constraint of \eqref{eq:Sampled_toy_2} when all scenarios in $ S $ are enforced. We then solve \eqref{eq:Sampled_toy_2} with all but the scenario removed in the previous stage being enforced, thus resulting in the scenario $ \delta^{(2)} = \argmax_{\delta \in S\setminus \delta^{(1)}} \delta$ to be discarded. We proceed similarly until $ r $ scenarios are removed. 
	
	Let $ x_{k}^\star(S) $ be the optimal solution at the $ (k+1)$-th stage of the removal procedure described in the previous paragraph. Note that $ x_{k}^\star(S) = \delta^{(k+1)} $, $ k = 0, \ldots, r $, where $ \delta^{(k+1)} $ represents the $ (k+1)$-th largest sample of $ S $. Our goal is to compute
	 the probability of constraint violation associated to the optimal solution of \eqref{eq:Sampled_toy_2} when the scenarios that belong to $ S \setminus \{ \delta^{(1)}, \ldots, \delta^{(r)} \} $ are enforced. Since $ \mathbb{P} $ is a uniform probability measure on $ [0,1] $, it is clear that, for each $ k \in \{0,\ldots,r\} $, such a probability is given by $ V(x_k^\star(S)) = \mathbb{P}\{ \delta \in \Delta: x_k^\star(S) < \delta\}  = 1 - x_k^\star(S)$. Observe that
\begin{align}
	\mathbb{P}^m &\{ (\delta_1, \ldots, \delta_m) \in \Delta^m: V(x^\star_k(S)) > \epsilon \}  \nonumber\\ &=\mathbb{P}^m \{ (\delta_1,\ldots,\delta_m) \in \Delta^m: \delta^{(k+1)} < 1-\epsilon \}\nonumber \\ &= \sum_{i = 0}^{k} {m \choose i} \epsilon^i (1-\epsilon)^{m-i}.
	\label{eq:vio_probability_dist}
\end{align}
	The first equality in \eqref{eq:vio_probability_dist} follows from the fact that $ x^\star_k(S) = \delta^{(k+1)} $ and the second by partitioning the space $ \Delta^m $ into $ k+1$ disjoint sets where each of these sets contains elements $ S $ for which exactly $ i $, $ i = 0, \ldots, k $, samples from the removed samples lie within the interval $ [1-\epsilon,1] $, i.e., exceeding $1-\epsilon$, and then applying the total law of probability. An alternative explanation using order statistics can be found in \cite{RMP:20c}. Hence, the distribution of the probability of constraint violation associated with the final solution of the considered removal strategy is given by $ V(x^\star_r(S)) $, which is obtained from \eqref{eq:vio_probability_dist} by substituting $ k = r $.
	
	\subsection{Two-dimensional example}
	Suppose that $m$ i.i.d. samples are drawn from a uniform distribution on the interval $ [0,1] $ and consider the scenario program that returns the minimum width interval containing the samples given by
	\begin{align}
		\minimise_{x,y \in [0,1],~y \geq x} & \quad y - x \nonumber \\
		\mathrm{subject~to} & \quad \delta  \in [x,y], \quad \delta \in S \setminus R
		\label{eq:Sampled_toy_3}.
	\end{align}
	For any collection of the samples $ S = \{ \delta_1, \ldots, \delta_m \} $, with $ \mathbb{P}^m$-probability one, the scenario program \eqref{eq:Sampled_toy_3} has a unique solution and, at the optimal solution, there are exactly two active constraints, namely, those associated with the smallest and the largest sample of $ S \setminus R $. Let $ r = 2 \ell < m, $ for some integer $\ell$, be an even integer and consider (similar as in the one dimensional example) a removal scheme that discards the active constraints of \eqref{eq:Sampled_toy_3} at each stage. Denote by $ x^\star_k(S) $, $  k = 0, \ldots, \ell $, the two dimensional vector containing the optimal solution of the $ (k+1)$-th stage. One of the components of $ x^\star_k(S) $ is the $ (k+1)$-th largest sample of $ S $, which we denote by $ \delta^{(k+1)} $, and the other the $ (k+1)$-th smallest sample of $ S $, which we denote by $ \delta_{(m-k)} $. 
	
	We are interested in the probability of constraint violation at the $ (k+1)$-th stage, which is given by
	\begin{align}
		V&(x^\star_k(S)) = \mathbb{P}\{ \delta \in \Delta: \delta \notin [\delta_{(m-k)},\delta^{(k+1)}] \} \nonumber \\
		&= \mathbb{P}\{ \delta \in \Delta: \delta < \delta_{(m-k)} \} + \mathbb{P}\{ \delta \in \Delta: \delta > \delta^{(k+1)} \}	\nonumber \\
		&= 1 - (\delta^{(k+1)} - \delta_{(m-k)} ) = 1 - L_{(k+1)}(S),
		\label{eq:vio_probability_2}
	\end{align}
where $ L_{(k+1)}(S) = (\delta^{(k+1)} - \delta_{(m-k)} ) $ represents the length of the interval after the removal of $ 2 k $ samples. Equation \eqref{eq:vio_probability_2} consists in the probability that a new sample is drawn from $ \mathbb{P} $ and it falls outside the interval  $ [\delta_{(m-k)},\delta^{(k+1)}] $. Let $ \epsilon \in [0,1] $, we have that 
	\begin{align}
		\mathbb{P}^m &\{ (\delta_1,\ldots,\delta_m) \in \Delta^m: V(x^\star_k(S)) > \epsilon \} \nonumber \\
		&= \mathbb{P}^m \{ (\delta_1,\ldots,\delta_m) \in \Delta^m: L_{(k+1)}(S) < 1 - \epsilon \}.
		\label{eq:toy_example_1}
	\end{align}
For each $ k \in \{ 0, \ldots, \ell \} $, let
	\begin{align}
		A_k &= \{ (\delta_1,\ldots,\delta_m) \in \Delta^m: \delta_{(m-k)} \leq \epsilon \} \nonumber \\
		B_k &= \{ (\delta_1,\ldots,\delta_m) \in \Delta^m:  V(x^\star_k(S)) > \epsilon  \},
		\label{eq:toy_example_2}
	\end{align}
where $ A_k $ contains the samples $ S $ whose $ (k + 1)$-th smallest element lies in the interval $ [0,\epsilon] $ and $ B_k $ contains the samples that lead to $ V(x^\star_k(S)) > \epsilon $. Using this notation, we can write \eqref{eq:toy_example_1} as 
	\begin{align}
		&\mathbb{P}^m \{ (\delta_1,\ldots,\delta_m) \in \Delta^m: V(x^\star_k(S)) > \epsilon \} = \mathbb{P}^m\{ B_k \}  \nonumber \\
		&= \mathbb{P}^m \{ A_k^c \cap B_k \} +  \mathbb{P}^m \{ A_k \cap B_k \} 
		\label{eq:toy_example_3}
	\end{align}
	where $ A^c $ stands for the set complement of $ A $. Let us analyze each of the terms in the right-hand side of \eqref{eq:toy_example_3} separately. We start with the first term. Note that
	\begin{align}
		&\mathbb{P}^m \{ A_k^c \cap B_k \} = \mathbb{P}^m \{ A_k^c \} \mathbb{P}^m \{ B_k | A_k^c \} = \mathbb{P}^m \{ A_k^c \} \nonumber \\ 
		&=\mathbb{P}^m \{ (\delta_1,\ldots,\delta_m) \in \Delta^m: \delta^{(k+1)} \leq 1-\epsilon \} \nonumber \\
		&= \sum_{i = 0}^k {m \choose i} \epsilon^i (1-\epsilon)^{m-i},
		\label{eq:toy_example_4}
	\end{align}
where the second equality follows from the fact that $  \mathbb{P}\{ B_k | A_k^c \} $ is equal to one due to $ A_k^c \subset B_k $, i.e., the length of $ [\delta_{(m-k)},\delta^{(k+1)}] $ is less than $ 1-\epsilon $ whenever the $ (k+1)$-th smallest sample in $ S $ is larger than $ \epsilon $; and the third equality due to the fact that $\mathbb{P}^m\{ (\delta_1,\ldots,\delta_m) \in \Delta^m: \delta_{(m-k)} > \epsilon \} = \mathbb{P}^m \{ (\delta_1,\ldots,\delta_m) \in \Delta^m: \delta^{(k+1)}  \leq 1 - \epsilon  \},$
\begin{figure}[!t]
	\centering
	\includegraphics[width=0.8\linewidth]{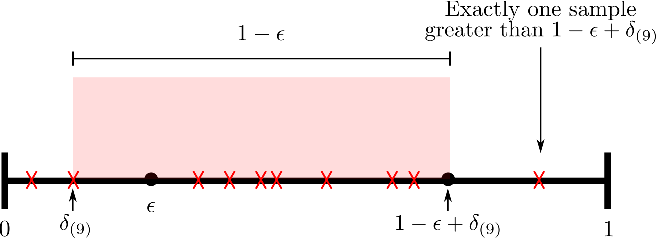}
	\caption{Realization of a sample $ S = \{\delta_1, \ldots, \delta_{10} \} $ that belongs to the set $ A_1 \cap B_1 \cap E_1^1 $ defined in \eqref{eq:toy_example_2} and \eqref{eq:toy_example_partition}.} 
	\label{fig:toy_example}
\end{figure}	
	which can be obtained by simple algebraic manipulations. Finally, the last equality holds due to \eqref{eq:vio_probability_dist}.
	
	To compute the second term in the right-hand side of \eqref{eq:toy_example_3}, it is convenient to define the partition of $ \Delta^m $ as
		\begin{align}
			E^j_k &= \left\{ (\delta_1,\ldots,\delta_m) \in \Delta^m: \text{ there are exactly } j \text{ samples } \right. \nonumber \\ 
			& \left. \hspace{2.5cm} \text{ greater than } 1 - \epsilon + \delta_{(m-k)} \right\},
			\label{eq:toy_example_partition}
		\end{align}
		where $ j = 0, \ldots, m-k $. Note that there can be no more than $ m-k $ samples greater than $ \delta_{(m-k)} $, as this would contradict the fact that $ \delta_{(m-k)} $ is the $ (k+1)$-th smallest sample of $ S $; hence, we have that $ \Delta^m = \cup_{j = 0}^{m-k} E^j_k $. For instance, Figure \ref{fig:toy_example} depicts a realization that belongs to $ A_1 \cap B_1 \cap E_1^1 $ when $ m = 10$, and exactly two discarded samples.
The partition given in \eqref{eq:toy_example_partition} allows us to write
	\begin{align}
		\mathbb{P}^m\{ &A_k \cap B_k \} = \sum_{j = 0}^{m-k} \mathbb{P}^m\{ A_k \cap B_k \cap E^j_k \} \nonumber \\ 
		&= \sum_{j = 0}^{k} \mathbb{P}^m\{ A_k \cap B_k \cap E^j_k \} + \sum_{j = k+1}^{m-k} \mathbb{P}^m\{ A_k \cap B_k \cap E^j_k \} \nonumber \\
		&= \sum_{j = 0}^{k} \mathbb{P}^m\{ A_k \cap B_k \cap E^j_k \},
		\label{eq:toy_example_6}
	\end{align}
	where the last equality follows from the fact that $ A_k \cap B_k \cap  E^j_k = \emptyset $ for all $ j > k $, since having more than $ k $ samples greater than $ 1-\epsilon+\delta_{(m-k)} $ with $ \delta_{(m-k)} \leq \epsilon $ implies that $ L_{(k+1)} \geq 1 - \epsilon $, so such a realization does not belong to $ B_k $. 
	
    We claim that
	\begin{align}
		\mathbb{P}^m&\{ A_k \cap B_k \cap E^j_k \} \nonumber \\ &= {m \choose j+k + 1} \epsilon^{j+k+1} (1-\epsilon)^{m-j-k-1}.
		\label{eq:toy_example_7}
	\end{align}
	Consider first the case where $ k=j=1 $. We can compute $ \mathbb{P}\{ A_1 \cap B_1 \cap E_1^1 \} $ as
	\begin{align*}	 
		\mathbb{P}^m\{ A_1 &\cap B_1 \cap E_1^1 \} = \int_{0}^{\epsilon}  \mathbb{P}^m\{ B_1 \cap E_1^1 | \delta_{(m-1)} = \xi \} d\xi  \nonumber \\ &= \int_{0}^{\epsilon} m \underbrace{ (m-1) (m-2) \xi (\epsilon - \xi)}_{T_1} \underbrace{(1-\epsilon)^{m-3}}_{T_2} d \xi,
	\end{align*}
	where the factor $m$ refers to the number of choices for $\xi$, the term $ T_1 $ is due to the fact that there must be exactly one sample less than $ \xi $ and exactly one sample greater than  $ 1-\epsilon+\xi $ and there are $ (m-1)(m-2) $ possible such samples once $\xi$ is chosen, and the term $ T_2 $ is due to the fact that the remaining samples ($ m-3 $ in this case) are within the interval $ [\xi,1-\epsilon+\xi] $. 
	
	To show \eqref{eq:toy_example_7} in general one may proceed inductively; alternatively, we can use the fact that the uniform distribution assigns the same probability to subsets of $ [0,1] $ that have the same length. Hence, $ \mathbb{P}^m\{A_k \cap B_k \cap E^j_k \} $ is the probability that $ j + k + 1 $ samples are outside the interval  $ [\delta_{(m-k)},1-\epsilon+\delta_{(m-k)}] $, and the remaining $ m - j - k - 1 $ ones to its complement. This immediately yields \eqref{eq:toy_example_7}.
	
	Combining \eqref{eq:toy_example_3}, \eqref{eq:toy_example_4}, \eqref{eq:toy_example_6} and \eqref{eq:toy_example_7}, we have that
	\begin{align}
		&\mathbb{P}^m\{ (\delta_1,\ldots,\delta_m) \in \Delta^m: V(x^\star_k(S)) > \epsilon \} = \mathbb{P}^m\{ B \} \nonumber \\
		&= \sum_{i = 0}^k {m \choose i} \epsilon^i (1-\epsilon)^{m-i} \nonumber \\ & \hspace{1.5cm}+ \sum_{j = 0}^k {m \choose j+k+1} \epsilon^{j+k+1} (1-\epsilon)^{m-j-k-1}  \nonumber \\
		&= \sum_{i = 0}^k {m \choose i} \epsilon^i (1-\epsilon)^{m-i} + \sum_{i = k+1}^{2k+1} {m \choose i} \epsilon^{i} (1-\epsilon)^{m-i} \nonumber
	 	\end{align}
	\begin{align} 
		&= \sum_{i = 0}^{2k + 1} {m \choose i} \epsilon^i (1-\epsilon)^{m-i}.
	\end{align}
	Set $ k=\ell $ and recall from the discussion following equation \eqref{eq:Sampled_toy_3} that $ r = 2\ell $, so we have proved that
	\begin{align}
		\mathbb{P}^m &\{ (\delta_1,\ldots,\delta_m) \in \Delta^m: V(x^\star_\ell(S)) > \epsilon \} \nonumber \\
		&= \sum_{i = 0}^{r + 1} {m \choose i} \epsilon^i (1-\epsilon)^{m-i}.
		\label{eq:toy_example_8}
	\end{align}
	
	These two examples show that the associated probability of constraint violation for the case where scenarios are discarded can be computed analytically and that the obtained probability is better than the one of Theorem \ref{theo:sampl-disc}. Motivated by this fact we will show a tighter bound on the probability of constraint violation for such scenario programs that is valid for an arbitrary $ d $, as long as the number of discarded scenarios is an integer multiple of $ d $. The only difference in a $ d $-dimensional example would be that the upper limit in the summation would be $ r+d-1  $ (See Theorem \ref{theo:main_result_tight}), which is consistent with \eqref{eq:vio_probability_dist} and \eqref{eq:toy_example_8} and is tighter than the bound of Theorem \ref{theo:sampl-disc}. This is due to the fact that these scenario programs satisfy Assumption \ref{assum:tight_result} (to be defined in the sequel), which is a sufficient condition to obtain such a tight bound.

\section{Proposed discarding scheme and main results} \label{sec:discard_scheme}
In this section, we formalize a removal scheme that results in a better bound on the probability of constraint violation of a scenario program with discarded constraints. For a given set of scenarios $S = \{ \delta_1, \ldots, \delta_{m} \} $, we solve a cascade of $ \ell + 1 $ optimization programs denoted by $ P_k,~ k \in \{ 0, \ldots, \ell \},$ where $ (\ell + 1)d < m  $. For each $k \in \{1,\ldots,\ell\}$, let
\begin{figure*}[ht]
	\includegraphics[width=\linewidth]{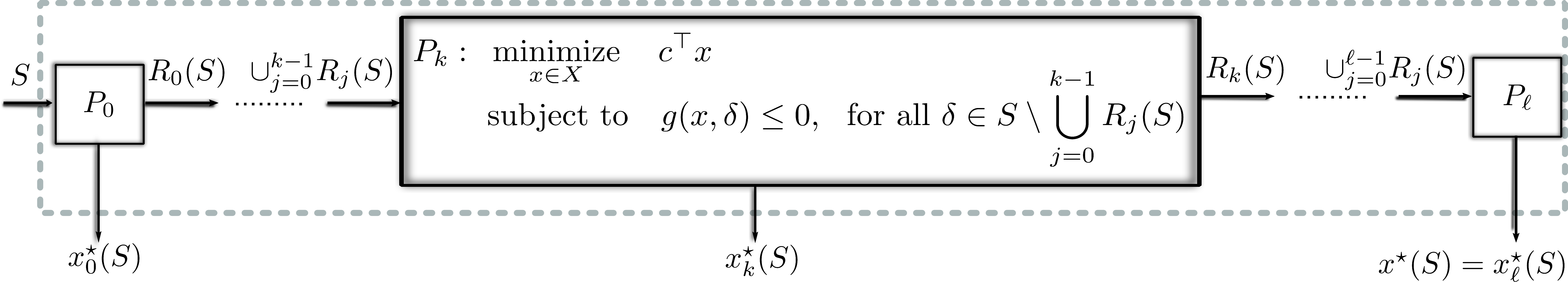}
	\centering
	\caption{Block diagram of the proposed scheme. Given $S = \{ \delta_1, \ldots, \delta_{m} \} $, with $ (\ell + 1)d < m $, we solve a cascade of $ \ell + 1 $ optimization programs denoted by $ P_k,~ k \in \{ 0, \ldots, \ell \} $, and remove $ R_k(S) $ scenarios with $|R_k(S)| = d$ at each stage. In total $ r = \ell d $ scenarios (the ones in $\bigcup_{j=0}^{\ell-1} R_j(S)$) are discarded.
		The set of discarded scenarios depends on the initial set $S$, thus we introduce it as argument of $R_k$.
		If each problem is fully-supported, then $R_k(S)$ corresponds to the (unique) support set associated with the minimizer $x_k^\star(S)$ of that program -- see~\eqref{eq:removed_constraint_fully_supp}; otherwise, $R_k(S)$ contains the support scenarios as well as additional scenarios selected according to a lexicographic order, as in~\eqref{eq:removed_non_deg}. The final solution is denoted by $x^\star(S) = x^\star_{\ell}(S) $. }
	\label{fig:proposed_scheme}
\end{figure*}
\begin{align*}
P_k:~ &\minimise_{x \in X \subset \mathbb{R}^d}  \quad c^\top x \\
&\mathrm{subject~to}  \quad g(x,\delta) \leq 0,~\text{ for all } \delta \in S \setminus \bigcup_{j=0}^{k-1} R_j(S),
\end{align*}
where $ R_k(S) $, with $ |R_k(S)| = d $, represents the set of removed scenarios at stage $ k $, and $ \bigcup_{j = 0}^{k-1} R_j(S) $ the ones that have been removed up to stage $ k $. For $ k = 0 $, we solve problem $ P_0 $ by enforcing all the scenarios in $ S $. Notice that the number of scenarios that have been removed up to stage $ \ell $ is given by $ \ell d $ (the samples in the set $ \bigcup_{j = 0}^{\ell - 1} R_j(S) $) and that, by construction, the collection of removed scenarios is disjoint. The choice of each set of discarded scenarios depends on the initial set $S$, thus we introduce it as an argument in $R_k$. A schematic illustration of the proposed scheme is provided in Figure \ref{fig:proposed_scheme}. Our choice for $R_k(S)$, $k\in \{0,\ldots,\ell-1\}$, will be detailed in the following two subsections.

\begin{defi}[Support set; see Definition 2 in \cite{CG:08}]
	Fix any $k \in \{0,\ldots,\ell\}$ and consider $P_k$. An element of $S \setminus \bigcup_{j=0}^{k-1} R_j(S)$ is a support scenario of $ P_k $, if its removal changes the minimizer $x_k^\star(S)$. The support set of $ x^\star_{k}(S) $, denoted by $ \mathrm{supp}(x^\star_{k}(S)) $, is the collection of support scenarios of $ S\setminus \cup_{j = 0}^{k-1} R_j(S) $.
	\label{def:supp_con}
\end{defi}

\begin{defi}[Fully-supported programs; see Definition 3 in \cite{CG:08}]
	Fix any $k \in \{0,\ldots,\ell\}$ and consider $P_k$. We say that $P_k$ is fully-supported if, for any $S$ with $ |S| = m $ and $m > d$, $|\mathrm{supp}(x_k^\star(S))| = d$ with $\mathbb{P}^m$-probability one.
	\label{def:fully_supp}
\end{defi}

\begin{defi}[Non-degenerate programs; see Assumption 2 in \cite{CG:18}]
	Fix any $k \in \{0,\ldots,\ell\}$ and consider $P_k$. We say that $P_k$ is non-degenerate if, with $\mathbb{P}^m$-probability one, solving the problem by enforcing the constraints only on the support set, $\mathrm{supp}(x_k^\star(S))$, results in $x_k^\star(S)$, i.e., the solution obtained when all samples in $S \setminus \bigcup_{j=0}^{k-1} R_j(S)$ are employed.
	\label{def:non_deg_assump}
\end{defi}

Note that if a problem is fully-supported then it is also non-degenerate, however, the opposite implication does not hold. Moreover, in a convex optimization context, non-degeneracy is a relatively mild assumption, and implies that scenarios give rise to constraints at general positions that do not have accumulation points. On the contrary, requiring a problem to be fully-supported is stronger, however, it exhibits interesting theoretical properties as, with $\mathbb{P}^m$-probability one, the number of support scenarios is exactly equal to $d$~\cite{CG:08,MPL:15}. 

In the sequel, we split our analysis into fully-supported and non-degenerate scenario programs. This facilitates our analysis as our proof construction is laid out better if we analyze the proposed removal scheme assuming that the scenario programs are fully-supported. In Section \ref{sec:discard_scheme_non_deg}, we lift this assumption and show how to extend the developed analysis to the more general case of non-degenerate scenario programs, that are typically encountered in the scenario approach literature.

\subsection{The fully-supported case}
\label{sec:discard_scheme_FS}
In this section we assume that, with $ \mathbb{P}^m $-probability one, the cardinality of the support set of problem $ P_k $, $ k = 0, \ldots, \ell $, is equal to $ d $. We formalize this in the following assumption.
 
\begin{assump}[Fully-supportedness]
	For all $ k \in \{ 0, \ldots, \ell \} $, $ P_k $ is fully-supported with $\mathbb{P}^m$-probability one. 
	\label{assump:fully_supp}
\end{assump}
 
Under Assumption~\ref{assump:fully_supp},  we let
\begin{equation}
R_k(S) = \mathrm{supp}(x^\star_{k}(S)), \quad  k \in \{ 0, \ldots, \ell - 1 \},
\label{eq:removed_constraint_fully_supp}
\end{equation}
 i.e., we remove the support set of the corresponding optimal solution of $ P_k. $ Note that the cardinality of $ R_k(S) $ is equal to $ d $ and this choice for the removed scenarios guarantees that the objective function decreases at each stage, thus improving performance. Moreover, for $ k = \ell $, we denote by $ R_{\ell}(S) $ the support set of $ x^\star_{ \ell }(S) $. Note that $ R_\ell(S) $ does not contain any removed scenarios. 
 
 Under Assumption~\ref{assump:fully_supp}, we obtain a tighter bound than that of Theorem~\ref{theo:sampl-disc}, as shown in the following theorem.

\begin{theo}
	Consider Assumptions \ref{assump:Feas_Uniq} and \ref{assump:fully_supp}. Fix $\epsilon \in (0,1)$, set $r=\ell d$ and let $m > r+d$. Consider also the scenario discarding scheme as encoded by \eqref{eq:removed_constraint_fully_supp} and illustrated in Figure \ref{fig:proposed_scheme}, and let the minimizer of the $\ell$-th program be $x^\star(S) = x_{\ell}^\star(S)$. We then have that
	\begin{align}
	\mathbb{P}^m \bigg\{ (\delta_1,\ldots,\delta_{m}) \in \Delta^m&: \mathbb{P} \big\{ \delta \in \Delta: g(x^\star(S),\delta) > 0 \big\}  > \epsilon  \bigg\} \nonumber  \\ 
	& \leq \sum_{i = 0}^{r+d-1} {m \choose i}\epsilon^i (1-\epsilon)^{m-i}.\label{eq:thm_sampl-disc_new}
	\end{align} 
	\label{theo:main_result_fully_supp}
\end{theo}

The proof of Theorem~\ref{theo:main_result_fully_supp} is deferred to Section~\ref{sec:main_proof_fully_supp}. It is important to note that inequality \eqref{eq:thm_sampl-disc_new} does not involve the combinatorial factor $ {r+d-1 \choose r} $ as in Theorem \ref{theo:sampl-disc}. For a fixed number of scenarios, probability of constraint violation, and confidence level, one is able to satisfy \eqref{eq:thm_sampl-disc_new} with a larger number of removed scenarios compared to \eqref{eq:thm_sampl-disc}. To see this, in Table \ref{tab:comparision}, we show the ratio between the number of removed scenarios obtained from Theorem \ref{theo:main_result_fully_supp} and that from Theorem \ref{theo:sampl-disc} (taken from \cite{CG:11}) for different values of $d$, as the violation level is fixed to $\epsilon = 0.05$ and the number of samples to $m = 40000$. Each entry of Table \ref{tab:comparision} is obtained by equating the right-hand side of the inequalities in both theorems to $\beta = 10^{-6}$, and using bisection to find the allowable number of removed scenarios $r$. The fact that we allow more constraints to be removed using the result of Theorem \ref{theo:main_result_fully_supp} (notice that this number increases with $d$) creates the potential of achieving a better cost, as the resulting problem is less constrained. The latter is, however, problem-dependent as both our removal scheme as well as the procedure adopted in \cite{CG:11} are not optimal. Numerical evidence in Section \ref{sec:exam} quantifies the potential cost improvement with our approach on a resource sharing example.

\begin{table}
	\centering
	\caption{Ratio between the number of removed samples allowed by Theorem \ref{theo:main_result_fully_supp} and Theorem \ref{theo:sampl-disc} for a fixed $\epsilon = 0.05, \beta = 10^{-6}, m = 40000$, and different values of $d$.}
		\label{tab:comparision}
	\begin{tabularx}{0.85\columnwidth}{c||ccccccc}
		$d $        & 10 &60   & 120  & 180  & 240  & 300  & 360    \\[1ex]
		\hline 
		$\mathrm{ratio}$& 1.18 & 1.63 & 2.04 & 2.42 & 2.59 & 3.21 & 3.62  \\     
	\end{tabularx}
\end{table}

To illustrate how the proposed scheme works, we consider the pictorial example of Figure~\ref{fig:pic-example}. Note that $ d = 2$, $ m = 6,$ and we remove $ r = 4,$ thus requiring $ \ell = 2 $ steps of the removal scheme of Figure~\ref{fig:proposed_scheme}. All the problems $ P_k, $ $ k \in \{ 0, 1, 2 \} $, are fully-supported, thus satisfying Assumption~\ref{assump:fully_supp}. The objective function is given by $ c^\top x = x_2 $ and is indicated by the downwards pointing arrow. The corresponding solution for the intermediate problems is illustrated by $ x^\star_k(S) $, for $ k \in \{ 0, 1, 2 \} $, and the support set of each stage by different colour patterns. For instance, the green constraints are the support set, namely, $ \mathrm{supp}(x^\star_0(S)) $, of problem $ P_0 $. The shaded colour under each constraint corresponds to the region of the plane that violates that given constraint, e.g., we notice that $ x^\star_1(S) $ violates both constraints that belong to $ \mathrm{supp}(x^\star_0(S)) $ and satisfies all the remaining ones. The result of Theorem~\ref{theo:main_result_fully_supp} provides guarantees for the probability of violation of $ x^\star_{2}(S) $. Note that the dashed-blue constraint is removed at stage $ 1 $, but it is not violated by the final solution of our scheme.

\begin{figure}[!t]
	\centering
	\includegraphics[width=0.8\linewidth]{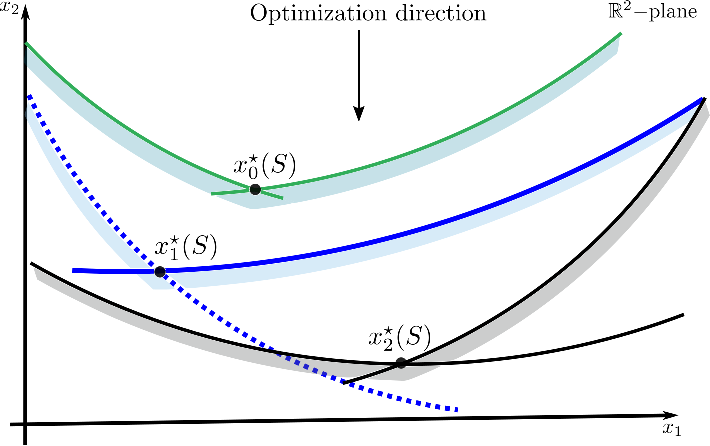}
	\caption{Pictorial example that illustrates the scheme proposed in Figure~\ref{fig:proposed_scheme} for fully-supported problems. In this case, we have that $ d = 2, m = 6, r = 4, $ and $ \ell = 2 $, and all the problems $ P_k, $ $ k \in \{ 0, 1, 2 \} $, satisfy Assumption~\ref{assump:fully_supp}. The objective function is given by $ c^\top x = x_2 $ (indicated by the downwards pointing arrow). The constraint sets are denoted by the different colors patterns: the green constraints are the samples in $ \mathrm{supp}(x^\star_0(S)) $, the blue ones in $ \mathrm{supp}(x^\star_1(S)), $ and the black ones in $ \mathrm{supp}(x^\star_2(S)) $. Note that the dashed-blue constraint is removed by the scheme of Figure~\ref{fig:proposed_scheme}, but it is not violated by $ x^\star_{2}(S) $.} 
	\label{fig:pic-example}
\end{figure}

\subsection{The non-degenerate case}
\label{sec:discard_scheme_non_deg}
In this subsection, we assume that problem $ P_k $, $ k \in \{ 0, \ldots, \ell \} $, is non-degenerate.
\begin{assump}[Non-degeneracy]
	For all $ k \in \{ 0, \ldots, \ell \} $, $ P_k $ is non-degenerate with $\mathbb{P}^m$-probability one. 
	\label{assump:non_deg_assump}
\end{assump}
In case of a non fully-supported problem ($\mathrm{supp}(x_k^\star(S)) <d$, for some $ k \in \{0, \ldots, \ell \} $), we adopt a procedure called regularization, in the same spirit as in~\cite{Cal:10}. This is based on introducing a lexicographic order as a tie-break rule to select which additional scenarios to append to $\mathrm{supp}(x_k^\star(S))$, thus constructing a set of cardinality $d$. 
Note that unless we impose such an order there is no unique choice as all scenarios that are not included in $\mathrm{supp}(x_k^\star(S))$ are not of support, hence their presence leaves the optimal solution unaltered.
To this end, we put a unique linear order on the elements of $S$, i.e., assigning them a distinct numerical label. For each $ k \in \{ 0, \ldots, \ell \} $, let $ \nu_k(S) = d - |\mathrm{supp}(x^\star_{k}(S))| $ and define recursively 
\begin{align}
&Z_k(S) = \bigg \{  \nu_k(S)  \text{ scenarios with the smallest labels in } \nonumber \\
& S \setminus \Big (\bigcup_{j = 0}^{k - 1} \big\{  \mathrm{supp}(x^\star_{j}(S)) \cup Z_j(S)  \big\} \cup \mathrm{supp}(x_k^\star(S))\Big )  \bigg \}, \label{eq:regular}
\end{align}
with $Z_0(S)$ containing the $\nu_0(S)$ smallest elements of $ S\setminus \mathrm{supp}(x_0^\star(S))$ according to the linear order. Note that the set appearing in the the definition of $Z_k(S)$ in \eqref{eq:regular} corresponds to scenarios available at stage $k$ that are not of support.

For each $k \in \{0,\ldots,\ell - 1\}$, we can now define the sets of discarded samples as
\begin{align}
R_k(S) = \mathrm{supp}(x^\star_{k}(S)) \cup Z_k(S). \label{eq:removed_non_deg}
\end{align}
Notice that by construction $|R_k(S)| = d$, while if for any $k \in \{0,\ldots,\ell\}$, $P_k$ is fully-supported, then $R_k(S) = \mathrm{supp}(x^\star_{k}(S))$, i.e., it coincides with the support set of $x^\star_{k}(S)$. Similar as in the fully-supported case, we denote by $ R_{\ell}(S) $ the superset of the support set of $ x^\star_{ \ell }(S) $ obtained by appending, if necessary, $ \nu_{\ell}(S) $ scenarios from the remaining ones. 

\begin{figure}[!t]
	\centering
	\includegraphics[width=0.8\linewidth]{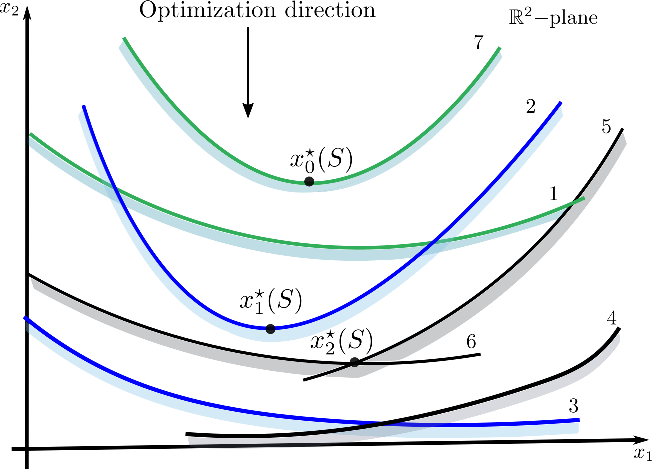}
	\caption{Illustration of the scheme in Figure~\ref{fig:proposed_scheme} applied to non-degenerate, but not fully-supported, problems. The intermediate solutions are denoted by $ x^\star_{ k }(S) $, $ k = 0, 1, 2 $. The different colour patterns depict the removed scenarios at each stage. The green constraints are the ones removed at stage $ 0 $, the blue ones those removed at stage $ 1 $. Similar as before, the objective function is given by $ c^\top x =  x_2 $ and this is indicated by the downwards arrow. Observe that the optimal solution, consequently the final solution returned by our scheme, depends on the linear order imposed to the original scenarios.} 
	\label{fig:pic_example_2}
\end{figure}

\begin{remark} \label{rem:regular_cost}
	Consider two arbitrary scenario sets $C \subset C'$, and denote by $x_k^\star(C)$ and $x_k^\star(C')$ the minimizers of $P_k$ with $C$ and $C'$, respectively, replacing $S$.  Moreover, define $Z_k(C)$ and $Z_k(C')$ as in \eqref{eq:regular} with $C$ and $C'$, respectively, in place of $S$. We then have that $ F_k(x_k^\star(C)) < F_k(x_k^\star(C')) $ if: either $ c^\top x_k^\star(C) < c^\top x_k^\star(C') $; or $ c^\top x_k^\star(C) = c^\top x_k^\star(C') $ and, at the first element that $Z_k(C)$ and $Z_k(C')$ differ, the corresponding label of $Z_k(C)$ is strictly lower with respect to the imposed lexicographic order than the one of 
	$Z_k(C')$. Regularization is thus a way to select among subsets of scenarios that would otherwise yield the same objective value. We will use this procedure in Section~\ref{sec:main_result_non_deg} to prove Theorem~\ref{theo:main_result_non_deg} below.	It is shown in \cite{Cal:10}, that $P_k$ with its objective function replaced by $F_k(x) = \big ( c^\top x, Z_k(S) \big )$ is a fully-supported program, and the constructed set $R_k(S)$ in~\eqref{eq:removed_non_deg} forms its unique support set with cardinality $d$.
\end{remark}

We are now in position to state the main result related to non-degenerate problems.

\begin{theo}
	Consider Assumptions \ref{assump:Feas_Uniq} and \ref{assump:non_deg_assump}. Fix $\epsilon \in (0,1)$, set $r=\ell d$ and let $m > r+d$. Consider also the scenario discarding scheme as encoded by \eqref{eq:removed_non_deg} and illustrated in Figure \ref{fig:proposed_scheme}, and let the minimizer of the $\ell$-th program be $x^\star(S) = x_{\ell}^\star(S)$. We then have that
	\begin{align}
	\mathbb{P}^m \bigg\{ (\delta_1,\ldots,\delta_{m}) \in \Delta^m&: \mathbb{P} \big\{ \delta \in \Delta: g(x^\star(S),\delta) > 0 \big\}  > \epsilon  \bigg\} \nonumber  \\ 
	& \leq \sum_{i = 0}^{r+d-1} {m \choose i}\epsilon^i (1-\epsilon)^{m-i}.\label{eq:thm_sampl-disc_new_2}
	\end{align} 
	\label{theo:main_result_non_deg}
\end{theo}

Theorem \ref{theo:main_result_non_deg} holds for non-degenerate scenario programs, thus being more general than Theorem \ref{theo:main_result_fully_supp}, which is only valid for fully-supported problems. This generalization comes at the expense of a (possible) decrease in performance, as we append additional scenarios to compose the support set of the regularized problem that may not improve the objective value associated to $ x^\star(S) $. However, similar to Theorem \ref{theo:main_result_fully_supp} we may still improve the cost with respect to other removal strategies, as we are allowed to remove more constraints compared to \cite{CG:11}.
It is also important to notice that \eqref{eq:thm_sampl-disc_new_2} holds for any linear order imposed in the original samples $ S $ and that the resulting optimal objective value of the scheme depends on such ordering. In this paper we only provide feasibility guarantees and not optimality. Note that this is also the case for the results in~\cite{CG:11,Cal:10}. The only available results for the optimal cost are given in~\cite{CG:11} when the removal scheme is the optimal one, which is, however, of combinatorial complexity. As a final remark, observe that Theorem \ref{theo:main_result_fully_supp} constitutes the main development towards obtaining Theorem \ref{theo:main_result_non_deg}, as should be apparent in the next section.

\begin{remark} \label{rem:viol_nondeg}
	It should be noted that the assumption in \cite{CG:11,Cal:10} appearing in the statement of Theorem \ref{theo:sampl-disc}, that requires all discarded scenarios to be violated by the final solution with $\mathbb{P}^m$-probability one, has some non-degeneracy implications for all intermediate problems. To see this, notice that if we allow for degenerate problems, then situations where all scenarios are identical are admissible and may happen with non-zero probability (allowing for atomic masses). Clearly, in such cases there is no scenario that can be discarded while being violated by the resulting solution which remains unaltered.
	Therefore, we tighten the bound in Theorem~\ref{theo:main_result_non_deg}, without strengthening the assumptions in \cite{CG:11,Cal:10}. 
\end{remark}

To clarify how the scheme presented in Figure~\ref{fig:proposed_scheme} works when applied to non-degenerate problems, consider the example depicted in Figure~\ref{fig:pic_example_2}. Similar as before, we have $d = 2, m = 7,$ and want to remove $ 4 $ constraints, i.e., $ r = 4$. As opposed to Figure~\ref{fig:pic-example}, however, note that the constraints are enumerated according to an arbitrary order, which is used to compose the sets $ Z_k(S) $, $ k \in \{ 0, 1,2 \} $, as described by equation~\eqref{eq:regular}. Moreover, problems $ P_0 $ and $ P_1 $ are not fully-supported, as the number of support scenarios is equal to one in each of these cases. Our scheme first removes the scenario that supports the solution $ x^\star_{0}(S) $ and the one labeled as $ 1 $, since it is the scenario with the smallest order among the remaining ones. These scenarios are depicted as green in Figure~\ref{fig:pic_example_2}. Then, we solve problem $ P_1 $ with the resulting scenarios, obtaining $ x^\star_1(S) $ as an intermediate solution and scenarios labeled as $ 2 $ and $ 3 $ to be removed. The former constraint is removed as it is in the support set of $ x^\star_1(S) $, and the latter as it is the sample with the smallest index from the remaining ones. Finally, the solution provided by the scheme, and whose guarantees are given in Theorem~\ref{theo:main_result_non_deg}, is denoted by $ x^\star_{2}(S) $.

 \section{Proof of the main Results}
 \label{sec:proof_main_results}


\subsection{The fully-supported case} \label{sec:main_proof_fully_supp}

Throughout this subsection, we consider Assumption~\ref{assump:fully_supp}. Let $m > (\ell+1)d$, and consider any set $C \subset S$, with $|C|=(\ell+1)d$.
We consider the proposed scheme of Figure~\ref{fig:proposed_scheme}, fed by $ C $ rather than $ S $. All quantities introduced in Section \ref{sec:discard_scheme} depending on $S$ would now depend on $C$ instead. 
For a given set of indices $ I \subset C $, we define  
\begin{align}
z^\star(I)=&\argmin_{x \in X}  \quad c^\top x \nonumber \\
&\mathrm{subject~to}  \quad g(x,\delta) \leq 0,~\text{for all}~\delta \in I. \label{eq:z_def}
\end{align}
Recall that $x_k^\star(C)$ denotes the minimizer of $P_k$ which in turn is based on the samples in 
$C \setminus \cup_{j = 0}^{k-1} R_j(C)$, i.e., the ones that have not been removed up to stage $k$ of the proposed scheme. It thus holds that $x_k^\star(C) = z^\star(C \setminus \cup_{j = 0}^{k-1} R_j(C))$ -- note that the argument of $ z^\star $ in this case depends on $ k $, $ k \in \{0,\ldots,\ell\} $. Recall also that, under Assumption~\ref{assump:fully_supp}, we have $ R_k(C) = \mathrm{supp}(x^\star_k(C)) $.

Since we will be invoking the framework introduced in Section~\ref{sec:learning}, let
\begin{align*}
	\mathcal{A}_1(C) &=  \big\{ \delta \in \Delta: g(x^\star_{\ell}(C), \delta) \leq 0 \big\}, \nonumber \\ 
\mathcal{A}_2(C) &=  \bigcap_{k = 0}^{\ell} \big\{ \delta \in \Delta: c^\top z^\star(J \cup \{\delta \}) \leq c^\top x^\star_{k}(C), \nonumber \\
&~~~~\text{for all } J \subset C \setminus \cup_{j = 0}^{k-1} R_j(C), \text{ with } |J| = d-1 \big\}, \nonumber \\
\mathcal{A}_3(C) &=   \bigcup_{k = 0}^{\ell-1} R_k(C), 
\end{align*}
and define the mapping $ \mathcal{A}: \Delta^m \rightarrow 2^\Delta $, with $ \zeta = (\ell+1)d $, as
\begin{align} 
\mathcal{A}&(C)= \Big( \mathcal{A}_1(C) \cap \mathcal{A}_2(C)\Big)  \cup \mathcal{A}_3(C).
\label{eq:main_algo_def}
\end{align}

The main motivation to define the mapping in~\eqref{eq:main_algo_def} is the fact that its probability of violation will be shown to upper bound that of $ \{ \delta \in \Delta: g(x^\star_{ \ell }(C), \delta) \leq 0 \} $, which is ultimately the quantity we are interested in (as shown in Section~\ref{sec:main_proof_fully_supp}, step~$ 3 $). 

Note that $ \mathcal{A}(C) $ comprises three sets:

\begin{enumerate}
	\item[$ i) $] The set $\mathcal{A}_1(C)$ contains all realizations of $\delta$ for which the final decision of our proposed scheme $x_\ell^\star(C) = x^\star(C)$ remains feasible. This is the set whose probability of occurrence we are ultimately interested to bound. 
	
	\item[$ ii) $] The set $\mathcal{A}_2(C)$, which is composed by the intersection of $\ell+1$ sets indexed by $k\in \{0,\ldots,\ell\}$, contains the realizations of $\delta$ such that, for all subsets of cardinality $ d -1 $ from the remaining samples at stage $ k $, the cost $c^\top z^\star(J \cup \{\delta \})$ is lower than or equal to $c^\top x^\star_{k}(C)$. The former cost corresponds to appending $\delta$ to any set $J$ of $d-1$ scenarios from $C \setminus \cup_{j = 0}^{k-1} R_j(C)$, while the latter corresponds to the cost of the minimizer $x_k^\star(C)$ of $P_k$. Informally, this inequality is of similar nature with that of $ \mathcal{A}_1(C) $, however, rather than considering constraint satisfaction it only involves some cost dominance condition for each of the interim and the final optimal solutions. The motivation to use this representation rather than constraint satisfaction conditions stems from the fact that in Section~\ref{sec:discard_scheme_non_deg} we will be appending a lexicographic order to the cost so that we break the tie among multiple compression sets. Besides, these sets carry information about the path taken by the proposed scheme, which is to be understood, in this context, as the sequence $ (x_k(C))_{k = 0}^\ell $. 
	
	\item[$ iii) $] The set $\mathcal{A}_3(C)$ includes all scenarios that are removed by the discarding scheme. Implicit in the definition of mapping~\eqref{eq:main_algo_def} is the fact that, for any compression set $ C $, all samples that are not removed in the intermediate stages must be contained in the set $ \mathcal{A}_1(C) \cap \mathcal{A}_2(C) $. This fact will be crucial in the following arguments. 
\end{enumerate}

The following proposition establishes a basic property of any compression associated to the mapping~\eqref{eq:main_algo_def}, and is instrumental for the proof of our main theorem.

\begin{prop} \label{prop:monotone}
	Consider Assumptions \ref{assump:Feas_Uniq} and \ref{assump:fully_supp}. Set $r=\ell d$ and let $m > (\ell+1)d$. We have that $ C \subset S$ is a compression set for $\mathcal{A}(C)$ in \eqref{eq:main_algo_def} if and only if, with $\mathbb{P}^m$-probability one, for all $k \in \{0,\ldots,\ell\}$,
	\begin{align}
	x^\star_{k}(C) = x^\star_{k}(S). \label{eq:opt_equal}
	\end{align} 
\end{prop}

\begin{proof}
	We first show necessity. Suppose that $ C $ is a compression set but, for the sake of contradiction, we have that there exists $ k \in \{ 0, \ldots, \ell \} $ and $ \bar{\delta} \in S \setminus C $ such that 
	\begin{equation}
	x_k^\star(C) \neq x^\star_{k}(C \cup \{\bar{\delta}\}).
	\label{eq:prop_theo1_1}
	\end{equation}
	Let $ \bar{k} $ be the minimum index such that~\eqref{eq:prop_theo1_1} holds, while we have that 	 
	$x^\star_j(C) = x^\star_{j}(C \cup \{ \bar{\delta} \})$, for all $j < \bar{k}$. 
	
	By Assumption~\ref{assump:fully_supp}, with $\mathbb{P}^m$-probability one, the last statement implies that $ \mathrm{supp}(x^\star_{j}(C)) = \mathrm{supp}(x^\star_{j}(C\cup\{\bar{\delta}\})) $, for all $j < \bar{k} $, as the support set of each optimal solution is unique. Hence, $R_j(C) = R_j(C \cup \{\bar{\delta}\})$ for all $j<\bar{k}$, and $R_j(C) = \mathrm{supp}(x_j^\star(C))$ for fully-supported problems (similarly for $R_j(C \cup \{\bar{\delta}\})$). By \eqref{eq:z_def}, we then have
	\begin{align}
	x_{\bar{k}}^{\star} (C) &= z^\star (C \setminus \cup_{j=0}^{\bar{k}-1} R_j(C)), \label{eq:bar_k_x_star_def} \\
	x^\star_{\bar{k}}(C\cup\{\bar{\delta}\}) &= z^\star ((C \setminus \cup_{j=0}^{\bar{k}-1} R_j(C)) \cup \{\bar{\delta}\}). \label{eq:bar_k_x_star_delta_def}
	\end{align}
	Since the right-hand side of \eqref{eq:bar_k_x_star_delta_def} involves one more scenario with respect to the right-hand side of \eqref{eq:bar_k_x_star_def}, the feasible set of~\eqref{eq:bar_k_x_star_delta_def} is a subset set of the one of~\eqref{eq:bar_k_x_star_def}. Moreover, by the fact that $ x^\star_{\bar{k}}(C \cup \{\bar{\delta}\}) \neq x^\star_{\bar{k}}(C)$ and Assumption~\ref{assump:Feas_Uniq}, we get 
	\begin{equation}
	c^\top x^\star_{ \bar{k} } (C) < c^\top x^\star_{ \bar{k} }(C \cup \{\bar{\delta}\}). \label{eq:cost_decrease}
	\end{equation}
	
	Notice that $\bar{\delta}$ belongs to the support set of $x^\star_{ \bar{k} }(C \cup \{\bar{\delta}\})$, as its removal results in a different optimal solution with lower cost in \eqref{eq:cost_decrease}. 
	In other words, there exists $ \bar{J} \subset C \setminus \cup_{j=0}^{\bar{k}-1} R_j(C)$ (in fact, $ \bar{J} = \mathrm{supp}(x^\star_{ \bar{k} }(C\cup\{\bar{\delta}\})) \setminus \{
	\bar{\delta}\}$) of cardinality $ d-1 $ such that by \eqref{eq:bar_k_x_star_delta_def}, we have that
	\begin{equation}
	c^\top x^\star_{ \bar{k} }(C) < c^\top x^\star_{ \bar{k} }(C \cup \{\bar{\delta}\}) = c^\top z^\star(\bar{J} \cup \{\bar{\delta}\}).
	\label{eq:intermediate_res_prop_1}
	\end{equation}
	
	At the same time, $C$ is assumed to be a compression set. Since $ \bar{\delta} \notin C $, then  $\bar{\delta} \notin \cup_{k = 0}^{\ell-1} R_k(C) = \mathcal{A}_3(C)$, as $\cup_{k=0}^{\ell-1} R_k(C) \subset C$. As a result, $\bar{\delta}$ will give rise to a constraint in $P_{\ell}$, hence $\bar{\delta} \in \mathcal{A}_2(C)$, which in turn implies that 
	for all $ J \subset C \setminus \cup_{j = 0}^{\ell-1} R_j(C) $ with $ |J| = d-1 $, and for all $ k \leq \ell $,
	\begin{equation}
	c^\top z^\star(J \cup \{ \bar{\delta} \}) \leq c^\top x^\star(C) \leq c^\top x^\star_{k}(C), \label{eq:compr_cost}
	\end{equation}
	where the first inequality follows from the fact that $c^\top x^\star(C)$ is the optimal value for $P_\ell$, and $x^\star(C)=x_{\ell}^\star(C)$ by construction satisfies all constraints with scenarios in $J \cup \{ \bar{\delta}\}$. The second inequality
	follows from the fact that $k\leq \ell$, and the cost deteriorates as $k$ increases.
	Setting $k=\bar{k}$ and $ J = \bar{J} $ in \eqref{eq:compr_cost} establishes a contradiction with \eqref{eq:intermediate_res_prop_1}, thus showing that $x_k^\star(C) = x_k^\star(C\cup \{\delta\})$, for any $\delta \in S\setminus C$, and any $k\in \{0,\ldots,\ell\}$. Inductively, adding one by one each element in $ S \setminus C $, we can show that $x_k^\star(C) = x_k^\star(S)$, for any $k\in \{0,\ldots,\ell\}$, thus concluding the necessity part of the proof.
	
	We now show sufficiency. Let $ C \subset S $ be such that $ x^\star_{k}(C) = x^\star_{k}(S) $ for all $ k \in \{0, \ldots, \ell\} $. We aim to show that $ C $ is a compression for $ S $, i.e., with $\mathbb{P}^m$-probability one, $ \delta \in \mathcal{A}(C) $ for all $ \delta \in S $. Recalling the definition of the mapping $ \mathcal{A}(C) $ from~\eqref{eq:main_algo_def} we note that, under this scenario, the sets $ \mathcal{A}_1(C) $ and $ \mathcal{A}_3(C) $ are trivially equal to $ \mathcal{A}_1(S) $ and $ \mathcal{A}_3(S) $, respectively. Moreover, since $ C \subset S $ and $ x^\star_k(C) = x^\star_k(S) $ for all $ k \in \{ 0, \ldots, \ell \} $, which implies that $ R_k(C) = R_k(S) $ by Assumption~\ref{assump:Feas_Uniq}, we have that $ S \setminus \cup_{j = 0}^{k-1} R_j(S) = S \setminus \cup_{j = 0}^{k-1} R_j(C) \supset C \setminus \cup_{j = 0}^{k-1} R_j(C) $. The latter implies then that the inequalities in $ \mathcal{A}_2(S) $ constitute a superset of those in $ \mathcal{A}_2(C) $, hence, that problem is more constrained and as a result $ \mathcal{A}_2(S) \subset \mathcal{A}_2(C) $. By construction we have that $\delta \in \mathcal{A}(S)$ for all $\delta \in S$. This in turn implies that if a sample is not removed, then it will have to be included in 
	$\mathcal{A}_2(S)$, and due to the established inclusion also in $\mathcal{A}_2(C)$. Since $\mathcal{A}_1(S) = \mathcal{A}_1(C)$ and $\mathcal{A}_3(S) = \mathcal{A}_3(C)$, we then have that $\delta \in \mathcal{A}(C)$ for all $\delta \in S$, showing that $C$ is a compression set. This concludes the proof of the proposition.
\end{proof}

\emph{Proof of Theorem~\ref{theo:main_result_fully_supp}:}
A natural compression candidate is
\begin{equation}
C = \bigcup_{k = 0}^{\ell} \mathrm{supp}(x^\star_{k}(S)),
\label{eq:compression_candidate}
\end{equation}
as it consists of the support sets of the intermediate problems.

\emph{Existence:} We prove that $ C $ in~\eqref{eq:compression_candidate} is a compression set. By the sufficiency part of Proposition~\ref{prop:monotone}, it suffices to show that, with $\mathbb{P}^m$-probability one, the set $ C $ in~\eqref{eq:compression_candidate} satisfies $x_k^\star(C) = x_k^\star(S)$, for all $k \in \{0,\ldots,\ell\}$. We will show this by means of induction. For the base case $k=0$, notice that 
\begin{align}
c^\top x_0^\star(S) = c^\top z^\star(S) &= c^\top z^\star(\mathrm{supp}(x_0^\star(S))) \nonumber \\
& = c^\top x_0^\star(C), \label{eq:cost_rank2}
\end{align}
where the first equality is due to \eqref{eq:z_def}, the second equality is due to the fact that $\mathrm{supp}(x_0^\star(S))$ is the support set of $ x^\star_0(S) $, while the last equality is due to Assumption~\ref{assump:fully_supp}, the definition of support set and the fact that $ \mathrm{supp}(x^\star_{0}(S)) \subset C $. By \eqref{eq:cost_rank2}, and Assumption~\ref{assump:Feas_Uniq}, we conclude that 
$x_0^\star(C) = x_0^\star(S)$.

To complete the induction argument, we assume that $ x^\star_{j}(C) = x^\star_{j}(S) $ for all $ j \in \{0,\ldots,\bar{k}\} $, for some $\bar{k} <\ell$. We will show that $ x^\star_{\bar{k} + 1}(C) = x^\star_{\bar{k} + 1}(S) $. To this end, by Assumption~\ref{assump:fully_supp}, $ x^\star_{j}(C) = x^\star_{j}(S) $ for all $ j \leq \bar{k} $ implies that $ \mathrm{supp}(x^\star_{j}(C)) = \mathrm{supp}(x^\star_{j}(S)) $, for all $j \leq \bar{k} $, as the support set of each optimal solution is unique. Moreover, $R_j(C) = R_j(S)$ for all $j<\bar{k}$, as $R_j(C) = \mathrm{supp}(x_j^\star(C))$ for fully-supported problems. Similarly to the base case we have that
\begin{align}
c^\top x_{\bar{k}+1}^\star(C) &= c^\top z^\star(C\setminus \cup_{j=0}^{\bar{k}} R_j(S)) \nonumber \\
&\leq c^\top z^\star(S\setminus \cup_{j=0}^{\bar{k}} R_j(S)) = c^\top x_{\bar{k}+1}^\star(S), \label{eq:cost_rank3}
\end{align}
where the first and last equalities are due to \eqref{eq:z_def}, and the inequality is due to the fact that $ C\setminus \cup_{j=0}^{\bar{k}} R_j(S) \subseteq S\setminus \cup_{j=0}^{\bar{k}} R_j(S).$

Moreover 
\begin{align}
c^\top &x_{\bar{k}+1}^\star(S) = c^\top z^\star(S\setminus \cup_{j=0}^{\bar{k}} R_j(S)) \nonumber \\ &= c^\top z^\star(\mathrm{supp}(x_{\bar{k}+1}^\star(S))) \leq c^\top z^\star(C \setminus \cup_{j=0}^{\bar{k}} R_j(S)) \nonumber \\&= c^\top x_{\bar{k} + 1}^\star(C), \label{eq:cost_rank4}
\end{align}
where the first and last equalities are due to \eqref{eq:z_def}, the second one due to the fact that $\mathrm{supp}(x_{\bar{k} + 1}^\star(S)) \subset S\setminus \cup_{j=0}^{\bar{k}} R_j(S)$, and the inequality holds since $R_j(C)=R_j(S)$ and $\mathrm{supp}(x_{\bar{k}+1}^\star(S)) \subset C\setminus \cup_{j=0}^{\bar{k}} R_j(S)$.
By \eqref{eq:cost_rank3} and \eqref{eq:cost_rank4} we then have that 
$x_{\bar{k}+1}^\star(C) = x_{\bar{k}+1}^\star(S)$, thus concluding the induction proof. In other words, we have shown that
\begin{align}
x_k^\star(C) = x_k^\star(S), \text{ for all } k \in \{0,\ldots,\ell\}. \label{eq:cost_cost}
\end{align}
Equation~\eqref{eq:cost_cost} together with the sufficiency part of Proposition~\ref{prop:monotone} shows that the candidate $ C $ in~\eqref{eq:compression_candidate} is a compression set. 

\emph{Uniqueness:} To show that $C$ in~\eqref{eq:compression_candidate} is the unique compression set, assume for the sake of contradiction that there exists another compression $C' \subset S$ for the mapping defined in~\eqref{eq:main_algo_def}, $ C' \neq C $, with $|C'|=(\ell+1)d$.
Since $C' \subset S$ is a compression, Proposition \ref{prop:monotone} (necessity part) implies that $x_k^\star(C')=x_k^\star(S)$, for all $k\in \{0,\ldots,\ell\}$, as $ C' $ is a compression. Besides, by the existence part (Step~$ 1 $ above), we have shown that for  $ C $ given in~\eqref{eq:compression_candidate} we have that $ x^\star_k(C)  = x^\star_{ k }(S)$ for all $ k \in \{ 0, \ldots, \ell \}.$ We thus have that for all $k\in \{0,\ldots,\ell\}$, $x_k^\star(C)=x_k^\star(C')$. This in turn implies that $ \mathrm{supp}(x^\star_{k}(C)) = \mathrm{supp}(x^\star_{k}(C')) $ for all $ k \in \{0, \ldots, \ell \} $, which, by Assumption \ref{assump:fully_supp}, leads to $C=C'$ (to see this notice that $ \cup_{k = 0}^{\ell} \mathrm{supp}(x^\star_{ k }(S)) \subset C' $ and $ |C'| = (\ell+1)d $), thus establishing a contradiction.

\emph{Linking Theorem~\ref{theo:Kostas_paper} with the probability of constraint violation:} Recall that
\begin{equation}
\mathcal{A}(C) = \Big(  \mathcal{A}_1(C) \cap \mathcal{A}_2(C) \Big) \cup \mathcal{A}_3(C),
\label{eq:main_res_structure_mapping}
\end{equation}
where the individual sets are as in~\eqref{eq:main_algo_def}. Recall also that $ \mathcal{A}_3(S) $ is a discrete set that contains the removed samples throughout the execution of the scheme of Figure~\ref{fig:proposed_scheme}. Fix any $S$ with $m$ scenarios, set $r=\ell d$ and let $m >(\ell+1)d$. Fix also $\epsilon \in (0,1)$.
Let $C \subset S$ with $|C|=(\ell+1)d$ be the unique compression defined in~\eqref{eq:compression_candidate}. We have that
\begin{align}
\mathbb{P}&\{\mathcal{A}(C)\} = \mathbb{P}\{ (\mathcal{A}_1(C) \cap \mathcal{A}_2(C))\cup \mathcal{A}_3(C)\} \nonumber \\
&=\mathbb{P}\{\mathcal{A}_1(C) \cap \mathcal{A}_2(C)\}, \nonumber \\
&\leq \mathbb{P}\{\mathcal{A}_1(C) \} = \mathbb{P}\{\delta \in \Delta:~ g(x^\star(C),\delta)\leq0 \}, \nonumber \\
&= \mathbb{P}\{\delta \in \Delta:~ g(x^\star(S),\delta)\leq0 \}, \label{eq:prob_ineq1}
\end{align}
where the first equality is due to the fact that $\mathbb{P}\{\mathcal{A}_3(C)\}=0$, since $\mathcal{A}_3(C)$ is a discrete set and we have imposed the non-degeneracy condition of Assumption \ref{assump:non_deg_assump} which prevents scenarios to have accumulation points with non-zero probability, while the inequality is due to the fact that $\mathcal{A}_1(C) \cap \mathcal{A}_2(C) \subseteq \mathcal{A}_1(C)$. The second last equality is by definition of $\mathcal{A}_1(C)$, and the last one follows from the fact that $x^\star(C) = x^\star(S)$ (see~\eqref{eq:cost_cost}).

We then have that if $\mathbb{P}\{\delta \in \Delta:~ g(x^\star(S),\delta)>0\} > \epsilon$ then $\mathbb{P}\{\delta \in \Delta:~ \delta \notin \mathcal{A}(C)\} > \epsilon$. As a result, $\{(\delta_1,\ldots,\delta_m)\in \Delta^m:~ \mathbb{P}\{\delta \in \Delta:~ g(x^\star(S),\delta)>0\} > \epsilon \} \subseteq \{(\delta_1,\ldots,\delta_m)\in \Delta^m:~ \mathbb{P}\{\delta \in \Delta:~  \delta \notin \mathcal{A}(C)\}  > \epsilon \}$. The last statement implies then that
\begin{align}
&\mathbb{P}^m \{(\delta_1,\ldots,\delta_m)\in \Delta^m:~ \mathbb{P}\{\delta \in \Delta:~ g(x^\star(S),\delta)>0\} > \epsilon \} \nonumber \\
&\leq \mathbb{P}^m\{(\delta_1,\ldots,\delta_m)\in \Delta^m:~ \mathbb{P}\{ \delta \notin \mathcal{A}(C)\}  > \epsilon \}.  \label{eq:prob_ineq2}
\end{align}
Therefore, since set $ C $ in \eqref{eq:compression_candidate} is the unique compression of $ \mathcal{A}(C) $, by Theorem~\ref{theo:Kostas_paper} we have that
\begin{align}
\mathbb{P}^m\{(\delta_1,\ldots,\delta_m)\in \Delta^m&:~ \mathbb{P}\{\delta \in \Delta:~  \delta \notin \mathcal{A}(C)\}  > \epsilon \} \nonumber \\
&\leq  \sum_{i = 0}^{r+d-1} {m \choose i}\epsilon^i (1-\epsilon)^{m-i}. \label{eq:prob_ineq3}
\end{align}
By \eqref{eq:prob_ineq2} and \eqref{eq:prob_ineq3} we then have that $\mathbb{P}^m \{(\delta_1,\ldots,\delta_m)\in \Delta^m:~ \mathbb{P}\{\delta \in \Delta:~ g(x^\star(S),\delta)>0\} > \epsilon \} \leq \sum_{i = 0}^{r+d-1} {m \choose i}\epsilon^i (1-\epsilon)^{m-i}$, thus concluding the proof of Theorem~\ref{theo:main_result_fully_supp}. \qed

\subsection{The non-degenerate case} \label{sec:main_result_non_deg}
 Throughout this subsection, we consider Assumption~\ref{assump:non_deg_assump}.
Let $m > (\ell+1)d$, and consider any set $C \subset S$ with $|C| = (\ell+1)d$.
We modify the mapping $\mathcal{A}(C)$ in 
\eqref{eq:main_algo_def} by replacing the second set in its definition with
\begin{align}  
&\mathcal{A}_2(C) = \bigcap_{k = 0}^{\ell} \big\{ \delta \in \Delta: F_k(z^\star(J \cup \{\delta \})) \leq F_k(x^\star_{k}(C)), \nonumber \\
&~~~~\text{for all } J \subset C \setminus \cup_{j = 0}^{k-1} R_j(C), \text{ with } |J| = d-1 \big\},
\label{eq:main_algo_def_mod}
\end{align}
where $F_k(\cdot)$ is the augmented objective function defined in Remark~\ref{rem:regular_cost}, related to $P_k$ defined by means of the regularization procedure of Section \ref{sec:discard_scheme}. The above inequality is to be understood in a lexicographic sense as detailed in Remark \ref{rem:regular_cost}.
A natural candidate compression set in this case is 
\begin{equation}
C = \bigcup_{k = 0}^{\ell} (\mathrm{supp}(x^\star_{k}(S)) \cup Z_k(S)),
\label{eq:compression_candidate_regular}
\end{equation}

which is composed by the removed samples of the scheme, and the support set of the last stage together with the corresponding constraints in $ Z_{\ell}(S) $. In fact, we now append $Z_k(S)$ in the definition of $C$ to ensure that $|C|=(\ell+1)d$, as $|\mathrm{supp}(x^\star_{k}(S))|$ could be lower than $d$ as the intermediate problems might not be fully-supported. Similarly to the fully-supported case, our goal is to show that the compression set defined in~\eqref{eq:compression_candidate_regular} is the unique compression set of size $ (\ell + 1)d $ for the mapping in~\eqref{eq:main_algo_def}, with $ \mathcal{A}_2(C) $ in~\eqref{eq:main_algo_def_mod} in place of $ \mathcal{A}_2(C) $ in~\eqref{eq:main_algo_def}.  By \eqref{eq:removed_non_deg}, recall that $R_k(C) = \mathrm{supp}(x^\star_{k}(C)) \cup Z_k(C)$, $k \in \{0,\ldots,\ell-1\}$. 

\begin{prop}
	Suppose Assumptions~\ref{assump:Feas_Uniq} and~\ref{assump:non_deg_assump} hold. Let $ C $ be the set in~\eqref{eq:compression_candidate_regular}, and consider the scheme of Figure~\ref{fig:proposed_scheme} with the removed scenarios given by \eqref{eq:removed_non_deg}. We have that, with $\mathbb{P}^m$-probability one, the following items hold:
	\begin{itemize}
		\item[$ i) $] $ x^\star_{ k }(C)= x^\star_{ k }(S) $ and $ Z_k(C) = Z_k(S) $ for all $ k \in \{ 0, \ldots, \ell \} $.
		
		\item[$ ii) $] Let $ C' $ be any other compression of size $ (\ell + 1)d $. Suppose $ R_j(C) = R_j(C') $ for all $ j \in \{ 0, \ldots, \bar{k}- 1 \} $, where $ \bar{k} $ is the smallest index such that $ x^\star_{\bar{k}}(C') \neq x^\star_{\bar{k}}(C) $. Then $ x^\star_{\bar{k}}(C') \neq x^\star_{\bar{k}} (C' \cup \{\delta\}) $ for some $ \delta \in \mathrm{supp}(x^\star_{\bar{k}}(C)) \setminus \mathrm{supp}(x^\star_{ \bar{k} }(C')) $. Moreover, such a $\delta$ is in fact in the set $ C \setminus C' $.
	\end{itemize}
\label{prop:non_deg_case}
\end{prop}

\begin{proof}
	\emph{Item $ i) $:} We use induction. Fix $ k = 0 $ and note that
\begin{equation}
x_0^\star(C) = z^\star(C) = z^\star(\mathrm{supp}(x^\star_0(S))) = x_0^\star(S), \label{eq:base_case_Prop_2}
\end{equation}
where the first equality follows from the definition in~\eqref{eq:z_def}, for the second one we use the definition of the support set, and the third one follows from the definition of $ x^\star_0(S) $ and the definition of the support set. Moreover, we have that 
\begin{align}
Z_0(C) &=  \Big\{ \nu_0(S)\text{ scenarios with the smallest labels in } \nonumber \\ & \quad \quad \quad \quad~~C \setminus \big\{ \mathrm{supp}(x^\star_0(S)) \big\}  \Big\} \nonumber\\ &=  \Big\{ \nu_0(S)\text{ scenarios with the smallest labels in } \nonumber \\ & \quad \quad \quad \quad~~S \setminus \big\{ \mathrm{supp}(x^\star_0(S)) \big\}  \Big\} = Z_0(S),
\end{align}
where the first equality is due to the definition of $ C $ in \eqref{eq:compression_candidate_regular} and the fact that $ Z_0(S) \subset C $, while the last one is due to the definition of $ Z_0(S) $ in \eqref{eq:regular}. Assume now that $ x^\star_{k}(C) = x^\star_{k}(S) $ and $ Z_k^\star(C) = Z^\star_{k}(S)  $ for all $ k \in \{ 0, \ldots, \bar{k} \} $, and consider the case $ \bar{k} + 1 $. Indeed, we have that
\begin{align}
x^\star_{\bar{k} + 1}(C) &= z^\star(C \setminus \cup_{j = 0}^{\bar{k}} R_j(C)) = z^\star(\mathrm{supp}(x^\star_{\bar{k} + 1}(S))) \nonumber \\ &=z^\star(S \setminus \cup_{j = 0}^{\bar{k}} R_j(S))=  x^\star_{\bar{k} + 1}(S),
\end{align}
where these relations follow as in \eqref{eq:base_case_Prop_2} for the case $ k = 0 $. We also have that
\begin{align}
&Z^\star_{\bar{k}+1}(C) = \Big\{  \nu_{\bar{k} + 1}(S)\text{ scenarios with the smallest labels in } \nonumber \\ &  C \setminus \Big( \bigcup_{j = 0}^{\bar{k}}  R_j (S) \cup \mathrm{supp}(x^\star_{\bar{k} + 1}(S)) \Big)
  \Big\} =  Z^\star_{\bar{k} + 1}(S),
\end{align}
since $ Z_{\bar{k} + 1}^\star(S) \subset C \setminus \bigcup_{j = 0}^{\bar{k}}  \{ R_j(S) \cup \mathrm{supp}(x^\star_{\bar{k} + 1}(S)) \} $ due to the particular choice of $ C $ in~\eqref{eq:compression_candidate_regular}, thus proving that for $ C $ in~\eqref{eq:compression_candidate_regular} we have $ x_k^\star(C) = x^\star_{k}(S) $ and $ Z_k(C) = Z_k(S) $, for all $ k \in \{ 0, \ldots, \ell \} $. This concludes the proof of item $ i) $.

\emph{Item $ ii) $:} 	We prove the contrapositive. Assume that for all $ \delta \in \mathrm{supp}(x^\star_{ \bar{k} }(C)) \setminus \mathrm{supp}(x^\star_{ \bar{k} }(C')) $ we have that $ x^\star_{ \bar{k} }(C') = x^\star_{ \bar{k} }(C' \cup \{\delta\} ).$ We will show that $ x^\star_{ \bar{k} }(C) = x^\star_{ \bar{k} }(C') $. We then have that
\begin{align}
c^\top x^\star_{ \bar{k} }(C') &= c^\top x^\star_{ \bar{k} }(C' \cup \{ \delta \}) = c^\top x^\star_{ \bar{k} }(C' \cup \mathrm{supp}(x^\star_{ \bar{k} }(C))) \nonumber \\
&= c^\top x^\star_{\bar{k}}(C),
\end{align}
where the second equality holds due to Lemma $ 2.12 $ in \cite{Cal:10} since $ C' \cup \{ \delta \} \subset C' \cup \mathrm{supp}(x^\star_{\bar{k}}(C)) $. The last equality follows from the definition of the support set and the non-degeneracy condition of Assumption 3. By Assumption~\ref{assump:Feas_Uniq} we then conclude that $ x^\star_{ \bar{k} }(C) = x^\star_{ \bar{k} }(C') $. 

We now show that such a $ \delta $ must belong to $ C \setminus C' $. In fact, choose $ \bar{\delta} \in \mathrm{supp}(x^\star_{\bar{k}}(C)) \setminus \mathrm{supp}(x^\star_{\bar{k}}(C'))  $ and assume for the sake of contradiction that $ \bar{\delta} \in C' $. This implies that $ \bar{\delta} \in R_j(C') $ for some $ j \geq \bar{k} $. In this is the case, we have that
\begin{align}
c^\top x^\star_{\bar{k}} (C' \cup \{ \bar{\delta} \}) &= c^\top z^\star \Big( \big(C'\setminus\bigcup_{j = 0}^{\bar{k} - 1} R_j(C)\big) \cup \{ \bar{\delta} \}\Big) \nonumber \\
&= c^\top z^\star \left( \mathrm{supp}(x^\star_{\bar{k}}(C')) \right)  \nonumber  \\ &= c^\top x^\star_{\bar{k}}(C') \label{eq:last_contradiction}
\end{align}
where the first relation holds due to~\eqref{eq:z_def} and the fact that $ R_j(C') = R_j(C) $ for all $ j < \bar{k} $, the second one is due to the fact that $ \mathrm{supp}(x^\star_{\bar{k}}(C')) \subset C'\setminus\bigcup_{j = 0}^{\bar{k} - 1} R_j(C) \cup \{ \bar{\delta} \}  $ and $ \bar{\delta} \in R_j(C') $ for $ j \geq \bar{k} $. The third equality follows from the definition of the support set and the non-degeneracy condition of Assumption~\ref{assump:non_deg_assump}.  However, note that~\eqref{eq:last_contradiction} contradicts our choice of $ \bar{\delta} $, which requires that $x_{\bar{k}}^\star(C') \neq x_{\bar{k}}^\star(C' \cup \{\bar{\delta}\})$. This concludes the proof.
\end{proof}
\emph{Proof of Theorem~\ref{theo:main_result_non_deg}: Existence.}  The existence part follows \textit{mutatis mutandis} from the one of Theorem~\ref{theo:main_result_fully_supp}. In fact, $\mathcal{A}_1(C) = \mathcal{A}_1(S)$ and $ \mathcal{A}_3(C) = \mathcal{A}_3(S) $ by Proposition~\ref{prop:non_deg_case}, item $ i) $, and $ \mathcal{A}_2(S) \subset \mathcal{A}_2(C) $ as $ C \subset S $ (see the discussion at the end of Proposition~\ref{prop:monotone}).
 \label{subsub:step2}
\emph{Uniqueness:} Let $ C' $ be another compression of size $ (\ell + 1)d $ and assume for the sake of contradiction that $ C \neq C' $. We can distinguish two possible cases. Case I: there exists a $ \bar{k} \in \{ 0, \ldots, \ell \} $ such that $ x^\star_{\bar{k}}(C') \neq x^\star_{\bar{k}}(C) $; or case II: $ x^\star_{k}(C') = x^\star_{k}(C) $ for all $ k \in \{ 0, \ldots, \ell \} $, but there exists a $ \tilde{k} \in \{ 0, \ldots, \ell \} $ such that $ Z_{\tilde{k}}(C') \neq Z_{\tilde{k}}(C) $. In the sequel, we argue separately that neither of these cases can happen.

\emph{Case I:} Let $ \bar{k} $ be the smallest index such that $ x^\star_{ \bar{k} }(C') \neq x^\star_{ \bar{k} }(C) $, and let $ \tilde{k} \leq \bar{k} $ be the smallest index such that $ Z_{\tilde{k}}(C') \neq Z_{\tilde{k}}(C).$ Consider first the case where $ \tilde{k} < \bar{k} $. Under these definitions, note that $ R_j(C') = R_j(C)  $ for all $ j < \tilde{k}$. Moreover, we have that
\begin{align}
Z_{\tilde{k}}(C') &= \Big\{  \nu_{\tilde{k}}(C)\text{ scenarios with the smallest labels in } \nonumber  \\ &  \quad \quad \quad \quad~~C' \setminus \bigcup_{j = 0}^{\tilde{k} - 1} R_j(C')  \Big\} \nonumber \\
&= \Big\{ \nu_{\tilde{k}}(C)\text{ scenarios with the smallest labels in} \nonumber \\ &  \quad \quad \quad \quad~~C' \setminus \bigcup_{j = 0}^{\tilde{k} - 1} R_j(S)  \Big\}, 
\end{align} 
where the first equality is by the definition in~\eqref{eq:regular} and the fact that $ \nu_{\tilde{k}}(C') = \nu_{\tilde{k}}(C) $ since $ \tilde{k} < \bar{k} $;  the second equality follows since $ R_j(C') = R_j(C) = R_j(S) $ -- the last equality follows from Proposition~\ref{prop:non_deg_case}, item $ i) $ --for all $ j \leq \tilde{k} - 1 $ and due to the uniqueness requirement of Assumption~\ref{assump:Feas_Uniq}. Note that $ Z_{\tilde{k}}(C') \neq Z_{\tilde{k}}(S) $ and $ C' \subset S $ implies that, for all $ \delta \in Z_{\tilde{k}}(C') \setminus Z_{\tilde{k}}(S) $,
\begin{equation}
y_\delta > \max_{\xi \in Z_{\tilde{k}}(S)} y_\xi = y_{\mathrm{max}},
\label{eq:relation_delta}
\end{equation}
where $ y_\delta \in \mathbb{N} $ corresponds to the label associated to $ \delta$.

We will use the relation~\eqref{eq:relation_delta} to show that any element in $ C \setminus C' $ has a label greater than $ y_{\mathrm{max}} $. In fact, note that
\begin{equation}
C' \setminus C \subset \left\{ \cup_{j = \tilde{k} + 1}^\ell R_j(C') \right\} \cup \left\{  Z_{\tilde{k}}(C') \setminus Z_{\tilde{k}}(C) \right\},
\label{eq:relation_delta_4}
\end{equation}
hence it suffices to show that any element in either set in the right-hand side of~\eqref{eq:relation_delta_4} is greater than $ y_{\mathrm{max}} $. To this end, fix any $ \delta \in \cup_{j = \tilde{k} + 1}^\ell R_j(C') $ and note that 
\begin{align}
 y_\delta &> \max_{\xi \in Z_{\tilde{k}}(C') \setminus Z_{\tilde{k}}(C)} y_\xi >y_{\mathrm{max}},
\label{eq:relation_delta_2}
\end{align}
where the first inequality is due to the fact that since such a $\delta$ has not been removed up to stage $\tilde{k}$, then its label will be greater than the ones in $Z_{\tilde{k}}(C')$, and as a result the ones in $Z_{\tilde{k}}(C') \setminus Z_{\tilde{k}}(C)$. The second inequality follows from \eqref{eq:relation_delta} and the fact that $Z_{\tilde{k}}(C') \setminus Z_{\tilde{k}}(C) \subset Z_{\tilde{k}}(S)$. Therefore, for any $ \delta \in C' \setminus C $ we have that $ y_{\delta} > y_{\mathrm{max}} $. 

From now on, let $ \delta $ be the scenario associated to $ y_{\mathrm{max}}. $ Pick $ \bar{J} = \{ \mathrm{supp}(x^\star_{\tilde{k}}(C)) \} \cup \{ Z_{\tilde{k}}(C) \setminus \{ \delta  \}\} $, which has cardinality $ d - 1 $ and is a subset of $ C \setminus \cup_{j=0}^{\tilde{k} - 1} R_j(C) $, and fix $ \bar{\delta} \in C' \setminus C $. Note that under this choice of $ \bar{\delta} $  
\begin{align}
F_{\tilde{k}}(z^\star(\bar{J} \cup \{ \bar{\delta} \})) >  F_{\tilde{k}} (x^\star_{\tilde{k}}(C)),
\label{eq:contradiction_bar_delta}
\end{align}
since $ y_{\bar{\delta}} > y_{\mathrm{max}} $ (by our previous discussion) and the inequality is interpreted lexicographically. However, this contradicts the fact that $ C $ is a compression set (see Definition \ref{def:compression_set}) as $ \bar{\delta} \in C' \setminus C \subset S $, hence $\bar{\delta} \notin \mathcal{A}_3(C')$ has not been removed, but $ \bar{\delta} \notin \mathcal{A}_2(C) $ due to~\eqref{eq:contradiction_bar_delta}.

Consider now the case $ \tilde{k} = \bar{k}.$ Note that, in this case, we have that $ R_j(C') = R_j(C) $ for all $ j \leq \bar{k} - 1 $. Based on the result of Proposition~\ref{prop:non_deg_case}, item $ ii) $, applied to $ C' $ (note that the assumptions of Proposition~\ref{prop:non_deg_case}, item $ ii) $, are satisfied with our choice of $ C' $), we observe that  there exists a $ \bar{\delta} \in  \{ \mathrm{supp}(x^\star_{k}(C)) \setminus \mathrm{supp}(x^\star_{k}(C')) \} \cap \{ C\setminus C' \}$ such that $ x^\star_{ \bar{k} }(C') \neq x^\star_{ \bar{k} }(C' \cup \{\bar{\delta}\}).  $ Repeating the arguments following equations~\eqref{eq:bar_k_x_star_def} and~\eqref{eq:bar_k_x_star_delta_def} in the necessity proof of Proposition~\ref{prop:monotone} with $ C' $ in the place of $ C $ in that proposition, we reach a contradiction that $ C' $ is a compression set. 

\emph{Case II:} We can reach a contradiction if case II holds in a similar fashion as in case I. In fact, letting $ \tilde{k} $ be the smallest index such that $ Z_{\tilde{k}}(C') \neq Z_{\tilde{k}}(C) $, the proof proceeds in an identical manner with case I.

Hence, we conclude that in any case $ C = C' $, thus proving uniqueness of the compression set in~\eqref{eq:compression_candidate_regular}. 

\emph{Linking Theorem~\ref{theo:Kostas_paper} with the probability of violation:} Note that for the non-degenerate case the mapping has the same structure as the one in~\eqref{eq:main_algo_def}, with the set $ \mathcal{A}_2(C) $ in~\eqref{eq:main_algo_def} being substituted with the one in~\eqref{eq:main_algo_def_mod}. The arguments then follows \textit{mutatis mutandis} the ones in the last part of the fully-supported case. This concludes the proof of Theorem~\ref{theo:main_result_non_deg}. \qed

\section{Tightness of the bound of Theorem~\ref{theo:main_result_fully_supp}} \label{sec:tight}
\subsection{Class of programs for which the bound is tight}
We provide a sufficient condition on the problems $ P_k $ so that the solution returned by the scheme of Figure~\ref{fig:proposed_scheme} achieves the upper bound given by the right-hand side of~\eqref{eq:thm_sampl-disc_new_2} when all the intermediate problems $ P_k, k = 0, \ldots ,\ell, $ are fully-supported. The result of this section implies that the bound of Theorem~\ref{theo:main_result_fully_supp} is tight, i.e., there exists a class of convex scenario programs where it holds with equality.

To this end, we replace the mapping $ \mathcal{A} $ in \eqref{eq:main_algo_def} with $ \mathcal{\bar{A}}: \Delta^m \rightarrow 2^\Delta $ defined
\begin{align}
\bar{\mathcal{A}}(C) = \Big\{ \delta \in \Delta: g(x^\star_{\ell}(C), &\delta) \leq 0 \Big\} \cup \nonumber \\ &\left\{  \bigcup_{k = 0}^{\ell - 1} \mathrm{supp}(x^\star_{k}(C)) \right\}.
\label{eq:def_mapping_tight_case}
\end{align}
Note that $ \bar{\mathcal{A}}(C) $ coincides with the one in \eqref{eq:main_res_structure_mapping}, but without the set $ \mathcal{A}_2(C) $ in its definition. We impose the following assumption.

\begin{assump}
	Fix any $S = \{ \delta_1, \ldots, \delta_{m} \} \in \Delta^m $ and let $ C \subset S $. For any $ k \in \{ 0, \ldots, \ell \} $ and $ \delta \in S$ such that $  \delta \in  \mathrm{supp}(x^\star_{k}(C)) $, we have that
	\[ 
	g(z^\star(J),\delta) > 0,	\]
	\label{assum:tight_result}
	for all $ J \subset C \setminus \big( \cup_{j=0}^{k-1} \mathrm{supp}(x^\star_{j}(C)) \cup \{ \delta \} \big) $ with $ |J| = d $.
\end{assump}

Assumption~\ref{assum:tight_result} imposes a restriction on the class of fully-supported problems. For instance, the pictorial example of Figure~\ref{fig:pic-example} does not satisfy Assumption~\ref{assum:tight_result}, even though all the intermediate problems $ P_k $ are fully-supported, as the dashed-blue removed constraint is not violated by the resulting solution. Indeed, Assumption~\ref{assum:tight_result} requires that, with $ \mathbb{P}^m $-probability one, whenever a sample belongs to the support scenarios of any intermediate problem, then the scenario associated with it is violated by all the solutions that could have been obtained using any subset of cardinality $ d $ from the remaining samples. Note that Assumption~\ref{assum:tight_result} is similar to the requirement of Theorem \ref{theo:sampl-disc}~\cite{CG:11,Cal:10}, however, in Theorem~\ref{theo:main_result_tight} below we exploit it in conjunction with the discarding scheme of Figure~\ref{fig:proposed_scheme} to show that the result of Theorem~\ref{theo:main_result_fully_supp} is tight. This serves as a constructive argument for the existential result of~\cite{CG:11}. In this paper we do not offer any means to check the validity of Assumption~\ref{assum:tight_result}; however, observe that the scenario programs studied in Section \ref{sec:analytic_sol} satisfy such an assumption.

\begin{theo}
	Consider Assumptions \ref{assump:Feas_Uniq}, \ref{assump:fully_supp}, and \ref{assum:tight_result}. Fix $\epsilon \in (0,1)$, set $r=\ell d$ and let $m > r+d$. Consider also the scenario discarding scheme as encoded by \eqref{eq:removed_constraint_fully_supp} and illustrated in Figure \ref{fig:proposed_scheme}, and let the minimizer of the $\ell$-th program be $x^\star(S) = x_{\ell}^\star(S)$. We then have that
	\begin{align}
	\mathbb{P}^m &\bigg\{ (\delta_1,\ldots,\delta_{m}) \in \Delta^m: \mathbb{P} \big\{ \delta \in \Delta: g(x^\star(S),\delta) > 0 \big\}  > \epsilon  \bigg\} \nonumber  \\ 
	& \hspace{2.8cm} = \sum_{i = 0}^{r+d-1} {m \choose i}\epsilon^i (1-\epsilon)^{m-i}.\label{eq:thm_sampl_tight}
	\end{align} 
	\label{theo:main_result_tight}
\end{theo}
\begin{proof}
	\emph{Existence:} We first show that the set $ C $ given in~\eqref{eq:compression_candidate} is a compression for the mapping in~\eqref{eq:def_mapping_tight_case}. Recall that under Assumption~\ref{assump:fully_supp} we have that $ R_k(S) = \mathrm{supp}(x^\star_{ k }(S))$ for all $ k \in \{ 0, \ldots, \ell \} $. Applying a similar induction argument as in the existence part of Theorem~\ref{theo:main_result_fully_supp}, we have that $ x^\star_{ k }(C) = x^\star_{ k }(S) $ for all $ k \in \{ 0, \ldots, \ell \}. $ Hence, by the definition of the mapping $ \bar{\mathcal{A}}(C) $ in~\eqref{eq:def_mapping_tight_case}, we obtain that $ \bar{\mathcal{A}}(C) = \bar{\mathcal{A}}(S) $, thus showing that $ C $ in~\eqref{eq:compression_candidate} is a compression.
	
	\emph{Uniqueness:} Let $ C' $ be another compression of size $ (\ell + 1)d $. We will show that $ x^\star_{ k }(C') = x^\star_{ k }(S) $ for all $ k \in \{ 0, \ldots, \ell \}, $ which by the existence part yields that $ x^\star_{ k }(C) = x^\star(C') $ for all $ k \in \{ 0, \ldots, \ell \}. $ By Assumption~\ref{assump:Feas_Uniq} and \ref{assump:fully_supp}, this would then imply that $ C = C' $.
	
To show that $ x^\star_{ k }(C') = x^\star_{ k }(S) $ for all $ k \in \{ 0, \ldots, m \}, $ it suffices to show that for all $ \delta \in S \setminus C' $ we have that 
\begin{equation}
	 x^\star_{ k }(C') = x^\star_{ k }(C' \cup \{ \delta \}), \text{ for all }  k \in \{ 0, \ldots, \ell \}.
	 \label{eq:intermediate_res_tight}
\end{equation}
In fact, if~\eqref{eq:intermediate_res_tight} holds for all $  \delta \in S \setminus C' $ by induction it follows then that $ x^\star_{ k }(C') = x^\star_{ k }(S) $ for all $ k \in \{ 0, \ldots, \ell \} $. 

To show~\eqref{eq:intermediate_res_tight} assume for the sake of contradiction that there exist a $ \bar{\delta} \in S \setminus C' $ and a $ k \in \{ 0, \ldots, \ell \} $ such that $ x^\star_{k}(C) \neq x^\star_{k}(C'\cup \{ \bar{\delta} \}) $. Let $ \bar{k} $ be the smallest index such that this occurs and note that 
\begin{equation}
x^\star_{ \bar{k} }(C') = z^\star(C'\setminus \cup_{j = 0}^{\bar{k} - 1} \mathrm{supp}(x^\star_{j}(C'))),
\end{equation}
\begin{equation}
x^\star_{ \bar{k} }(C'\cup \{ \bar{\delta} \}) = z^\star((C'\setminus \cup_{j = 0}^{\bar{k} - 1} \mathrm{supp}(x^\star_{j}(C'))\cup \{\bar{\delta} \}),
\label{eq:proof_theo4_2}
\end{equation}
which implies that $ \bar{\delta} \in \mathrm{supp}(x^\star_{ \bar{k} }(C' \cup \{\bar{\delta}\})) $, as removal of $\bar{\delta}$ will change $x_{\bar{k}}^\star(C' \cup \{\bar{\delta}\})$ to $x_{\bar{k}}^\star(C')$. By Assumption~\ref{assum:tight_result} and since $\mathrm{supp}(x_j^\star(C'))= \mathrm{supp}(x_j^\star(C' \cup \{\bar{\delta}\}))$ for all $j=0,\ldots,\bar{k}-1$, we have that for all $ J \subset C'  \setminus \big( \cup_{j = 0}^{\bar{k} - 1} \mathrm{supp}(x^\star_{j}(C')) \cup \{ \bar{\delta} \} \big)$ with cardinality $ d$,
\begin{equation}
g(z(J),\bar{\delta}) > 0.
\end{equation}
Hence, since $ \bar{J} = \mathrm{supp}(x^\star_{ \ell }(C')) $ is a subset of cardinality $ d $ of $ C' \setminus \big( \cup_{j = 0}^{\bar{k}-1} \mathrm{supp}(x^\star_{j}(C')) \cup \{\bar{\delta}\} \big)$, as these constraints have not been removed from $C'$, we obtain that
\begin{equation}
g(z(\bar{J}),\bar{\delta}) = g(x^\star_{\ell}(C'),\bar{\delta}) > 0,
\label{eq:prop_theo4_3}
\end{equation}
where the equality follows from~\eqref{eq:z_def}.  However, $C'$ is assumed to be a compression set for $\bar{\mathcal{A}}$, which implies that $\delta \in \bar{\mathcal{A}}(C')$, i.e., $g(x_\ell^\star(C'), \bar{\delta}) \leq 0$. This is in contradiction with~\eqref{eq:prop_theo4_3}, implying that $x_k^\star(C') = x_k^\star(C' \cup \{\delta\})$, for any $k \in \{0,\ldots,\ell\}$, for any $\delta \in S \setminus C'$. Using induction, adding one by one $ \delta \in S \setminus C' $, we can then show that $ x^\star_{ k }(C') = x^\star_{ k }(S) = x^\star_{ k }(C) $ for all $ k \in \{ 0, \ldots, \ell \},  $ thus showing that $ C $ in~\eqref{eq:compression_candidate} is the unique compression set for the mapping defined in~\eqref{eq:def_mapping_tight_case}. 

By Theorem~\ref{theo:Kostas_paper}, we then have that
\begin{align}
&\mathbb{P}^m\{ (\delta_1, \ldots, \delta_m) \in \Delta^m: \mathbb{P}\{ \delta: \delta \notin \bar{\mathcal{A}}(C)\}  > \epsilon  \} \nonumber \\ &= \mathbb{P}^m\{ (\delta_1, \ldots, \delta_m) \in \Delta^m: \mathbb{P}\{ \delta: g(x^\star_{ \ell }(C), \delta) > 0\}  > \epsilon  \} \nonumber \\ 
 &= \mathbb{P}^m\{ (\delta_1, \ldots, \delta_m) \in \Delta^m: \mathbb{P}\{ \delta: g(x^\star_{ \ell }(S), \delta) > 0\}  > \epsilon  \} \nonumber \\
 &= \sum_{i = 0}^{r + d - 1} {m \choose i} \epsilon^i (1-\epsilon)^{m - i},
\end{align}
where the first equality follows since the union of support scenarios is a discrete set and will be of measure zero. To obtain the second equality we have used the fact that $ x_\ell^\star(C) = x^\star_{ \ell }(S) $ for the compression set defined in~\eqref{eq:compression_candidate}. This concludes the proof of Theorem~\ref{theo:main_result_tight}.
\end{proof}

\section{Numerical example} 
\label{sec:exam}

In this section, we consider a resource allocation problem to illustrate our theoretical result. Suppose that a manufacturer produces a good in $ d $ different locations, and that this good can be produced from $ n $ different resources. The quantity of resource $ p $, $ p = 1, \ldots, n $, that is needed to produce a unitary amount of the given good at facility $ j $, $ j = 1, \ldots, d $, is a random variable parametrized by $\delta \in \mathbb{R}$, and is denoted by $ a_{pj}(\delta) $. We assume that the amount of resources $ p $ available to all facilities is deterministic. The  objective is to maximize the production, given by $ \sum_{j = 1}^d x^j $, where $ x^j $ is the $ j-$th component of $ x \in \mathbb{R}^d $, while keeping the risk of running out of resources under control.

Under the scenario theory we do not have access to the distribution that generates $ a_{pj}(\delta), p = 1, \ldots, n, j = 1, \ldots, d $; however, we encode it by means of data $ (a_{pj}(\delta_i))_{i = 1}^m $ for all $ p = 1, \ldots, n $ and for all $ j = 1, \ldots, d $, and solve the following fully-supported convex scenario problem
\begin{align}
\minimise_{\{ x^j \geq 0\}_{j = 1}^d} \quad & c^\top x \nonumber \\
\mathrm{subject~to} & \quad A(\delta_i) x \leq b,~\text{ for all } i =1, \ldots,m, 
\label{eq:logistic_problem}
\end{align}
where, for each $ i \in \{1, \ldots, m \} $, $ A(\delta_i) \in \mathbb{R}^{n \times d} $ is a matrix whose $ (p,j) $-th entry is given by $ a_{pj}(\delta_i) $, $ b \in \mathbb{R}^n $ is a vector whose $ p-$th component is the amount of resource $ p $ available to all facilities, and $ c = \begin{bmatrix}
-1 & \ldots & -1
\end{bmatrix}^\top \in  \mathbb{R}^d. $

\begin{figure}[!t]
	\centering
	\includegraphics[width=0.85\linewidth]{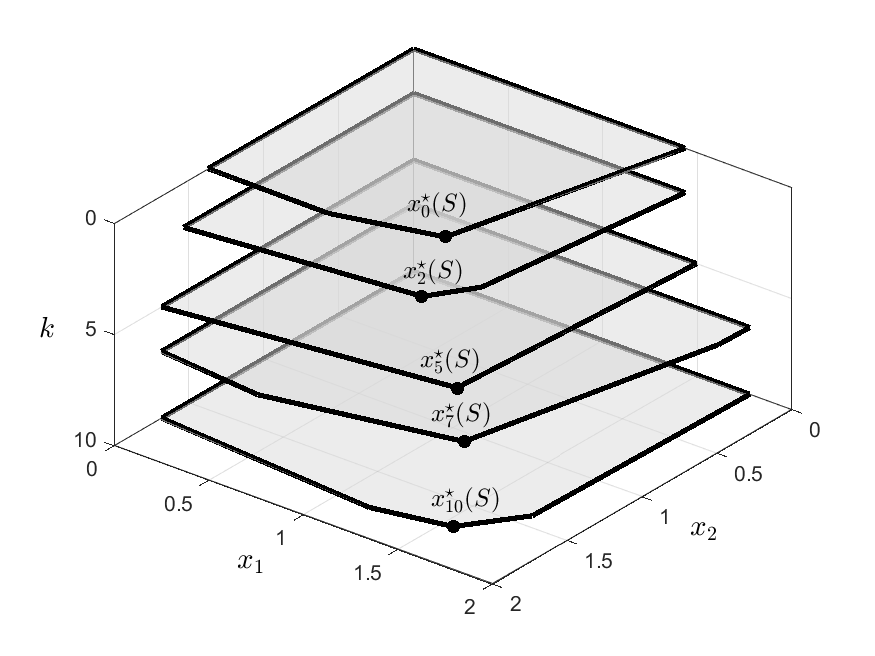}\\
	\caption{Feasibility sets of the intermediate problems $ P_k $, $ k = 0, 2, 5, 7, $ $ 10 $, for the scheme proposed in Figure~\ref{fig:proposed_scheme} when applied to \eqref{eq:logistic_problem}. The optimal solution of each problem is denoted by $ x^\star_{ k }(S) $, $ k = 0, 2, 5, 7 $, $ 10. $}
	\label{fig:2D_logistic_constraints}
\end{figure} 

\subsection{The two-dimensional case}

Set $ d = 2 $ and consider $ 2000 $ scenarios from the unknown distribution\footnote{For our simulations, fix $ i \in \{ 1, \ldots, m \} $ and set $ A(\delta_i) = 0.04 B(\delta_i) $, where $ B(\delta_i) \in \mathbb{R}^{n \times d} $, with entries obtained from a Laplacian distribution with mean equal to one and variance equal to three. Our numerical results were obtained setting the ``seed'' equal to $ 30 $ in \textsc{Matlab}.} for $ \delta $. We study the behavior of the scheme in Figure~\ref{fig:proposed_scheme} when we discard $ r = 20 $ of these scenarios. In this case, note that according to the description given in Section~\ref{sec:discard_scheme}, we have to solve a cascade of $ 11 $ optimization problems (i.e, $ \ell = 10 $ in the scheme of Figure~\ref{fig:proposed_scheme}).

\begin{figure}[!t]
	\centering
	\includegraphics[width=0.85\linewidth]{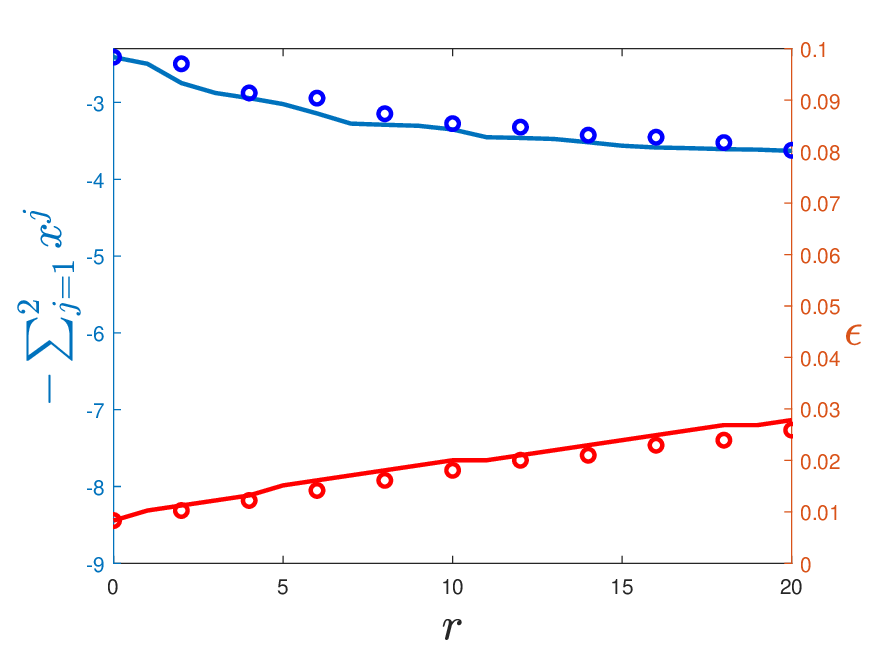}
	\caption{Cost and probability of constraint violation for the solution returned by the scheme of Figure~\ref{fig:proposed_scheme} and a greedy removal strategy for the problem in \eqref{eq:logistic_problem} when $ d = 2$, $ n = 2 $, and $ m = 2000 $. With the blue dots we show the cost obtained by the proposed procedure, where we are allowed to remove scenarios in batches of $d=2$, while the solid one shows the performance obtained by the greedy removal strategy where scenarios are removed one by one. In red we show the behavior of the probability of constraint violation.} 
	\label{fig:2D_logistic_constraints_2}
\end{figure} 
 
Figure~\ref{fig:2D_logistic_constraints} illustrates the feasible set for stages $ k = 0, 2, 5, 7, $ and $ 10 $ of the scheme of Figure~\ref{fig:proposed_scheme}, and depicts the corresponding optimal solution for each $ P_k  $ as $ x^\star_{ k }(S) $. Note that the feasible set associated to each problem $ P_k $ grows as we remove scenarios. To complement this analysis, we also show in Figure~\ref{fig:2D_logistic_constraints_2} a comparison between our method and the greedy scenario removal strategy as described in~\cite{CG:11}, which removes scenarios one by one according to one that yields the best improvement in the cost.  With the blue dots we show the cost obtained by the proposed procedure, where we are allowed to remove scenarios in batches of $d=2$, while the solid one shows the performance obtained by the greedy removal strategy, where scenarios are removed one by one. In red we show the corresponding behavior of the probability of constraint violation. This is calculated from the bounds of Theorem~\ref{theo:main_result_fully_supp} and Theorem~\ref{theo:sampl-disc}, respectively, with $\beta =  10^{-6}$.

\subsection{The 10-dimensional case}

Consider now~\eqref{eq:logistic_problem} with $d=10$ and the same $ 2000 $ scenarios. We compare the cost improvement of the proposed bound (Theorem~\ref{theo:main_result_non_deg}) with the one of Theorem~\ref{theo:sampl-disc}~\cite{CG:11}. To this end, for a given $ \epsilon \in [0.01,0.08] $, we compute the maximum number of scenarios that can be removed using each of these bounds. Note that due to the fact that we remove scenarios  in of $d$, we compute the number of scenarios that need to be removed by means of numerical inversion from the bound of Theorem 4 (using $m=2000$, $\beta = 10^{-6}$ and the given $
\epsilon$), and round it down to the closest multiple of $d=10$. For instance, for $ \epsilon = 0.03 $ the maximum number of scenarios that can be removed using the bound in~\eqref{theo:main_result_non_deg} is $ r = 18 $, but we only remove $ 10 $. Figure~\ref{fig:10D_logistic_2} shows then the relative cost difference $100 \times \frac{f^\star(\epsilon) - \bar{f}^\star(\epsilon)}{\bar{f}^\star(\epsilon)} $ as a function of $ \epsilon $, where $ f^\star(\epsilon) $ is the optimal value of problem~\eqref{eq:logistic_problem} when scenarios are removed according to Theorem~\ref{theo:main_result_non_deg}, and $ \bar{f}^\star(\epsilon) $ correspond to the bound in~\cite{CG:11}. For $\epsilon > 0.03$, our scheme leads to better optimal costs (i.e., the relative cost difference is negative), 
achieving approximately $ 4 \% $ of improvement when $ \epsilon = 0.08 $. This is due to the fact that more scenarios can be removed, while guaranteeing the same level of violation. Notice also that  there is no improvement when $ \epsilon \leq 0.03 $. This can be explained by two reasons: (1) due to the limitation on the number of removed scenarios, the proposed bound returns a value for $ r $ that is less than $ 10 $, hence no scenarios are removed (this is the case for $\epsilon \in [0.01,0.02]$);  or (2) the scenarios discarded by the greedy strategy lead to a better cost improvement (which happens for the case where $\epsilon = 0.03$).

\begin{figure}[!t]
	\includegraphics[width=0.85\linewidth]{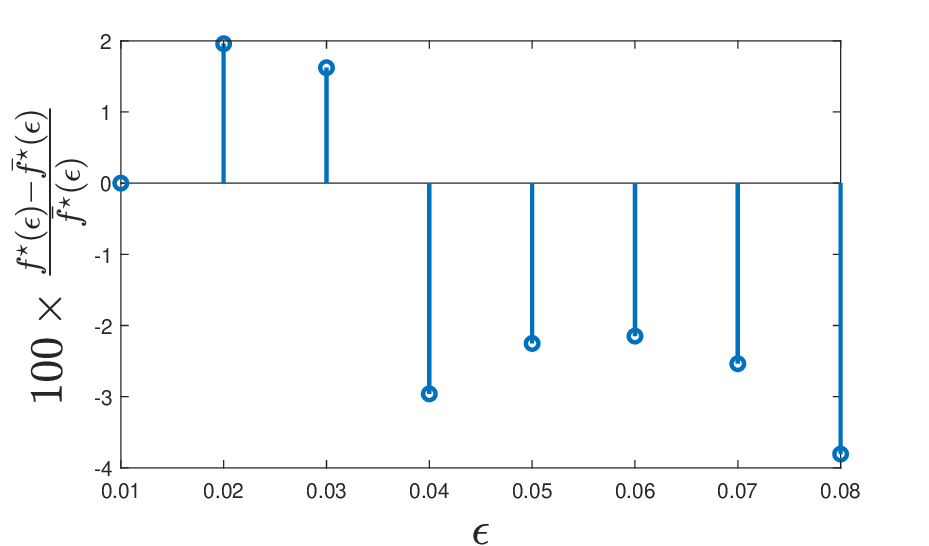}
		\caption{Relative cost improvement $100 \times \frac{f^\star(\epsilon) - \bar{f}^\star(\epsilon)}{\bar{f}^\star(\epsilon)} $, as a function of $\epsilon$, where $f^\star(\epsilon)$ corresponds to the cost associated with Theorem~\ref{theo:main_result_non_deg}, and $\bar{f}^\star(\epsilon)$ to the one of Theorem~\ref{theo:sampl-disc},~\cite{CG:11}. The numerical results correspond to \eqref{eq:logistic_problem} with $d=10$.}
	\label{fig:10D_logistic_2}
\end{figure}

Even though improving the computational requirements of the discarding procedure is not the main focus of our work, as a byproduct of the proposed removal scheme, the computational requirements of the proposed approach are lower with respect to the greedy removal strategy in~\cite{CG:11} (see also \cite{Cal:10}). To put this in perspective, to remove $ 100 $ scenarios in the previous example when $ d = 10 $, the greedy strategy requires the solution of $ 1101 $ optimization problems of the form~\eqref{eq:logistic_problem}, whereas the proposed scheme only needs to solve $ 11 $ of these problems. The computational savings are more pronounced as the dimension of the problem grows.  However, the performance improvement of the proposed discarding scheme with respect to the greedy removal strategy described in~\cite{CG:11,Cal:10} is problem dependent in general.

\section{Concluding remarks} \label{sec:concl}

In this paper we proposed a scenario discarding scheme that consists of a cascade of optimization problems, where at each stage we remove a superset of the support constraints. By relying on results from compression learning theory, we provide a less conservative bound on the probability of constraint violation of the obtained solution. Besides, we show that the proposed bound is tight, and characterize a class of problems for which this is the case. 

Current work concentrates towards extending our scenario discarding scheme so that we no longer remove scenarios in batches but one by one. Preliminary results in this direction can be found in \cite{RMP:21}. We also aim at exploiting the dual variables associated with each constraint in order to create a tie-break rule to choose the scenarios to be removed at each stage. 

\section*{Acknowledgement}
The authors are grateful to Professors Simone Garatti and Marco Campi for interesting discussions. Special thanks to Prof. Garatti for several insightful technical comments.

\balance



\begin{IEEEbiography}[{\includegraphics[width=1in,height=1.25in,clip,keepaspectratio]{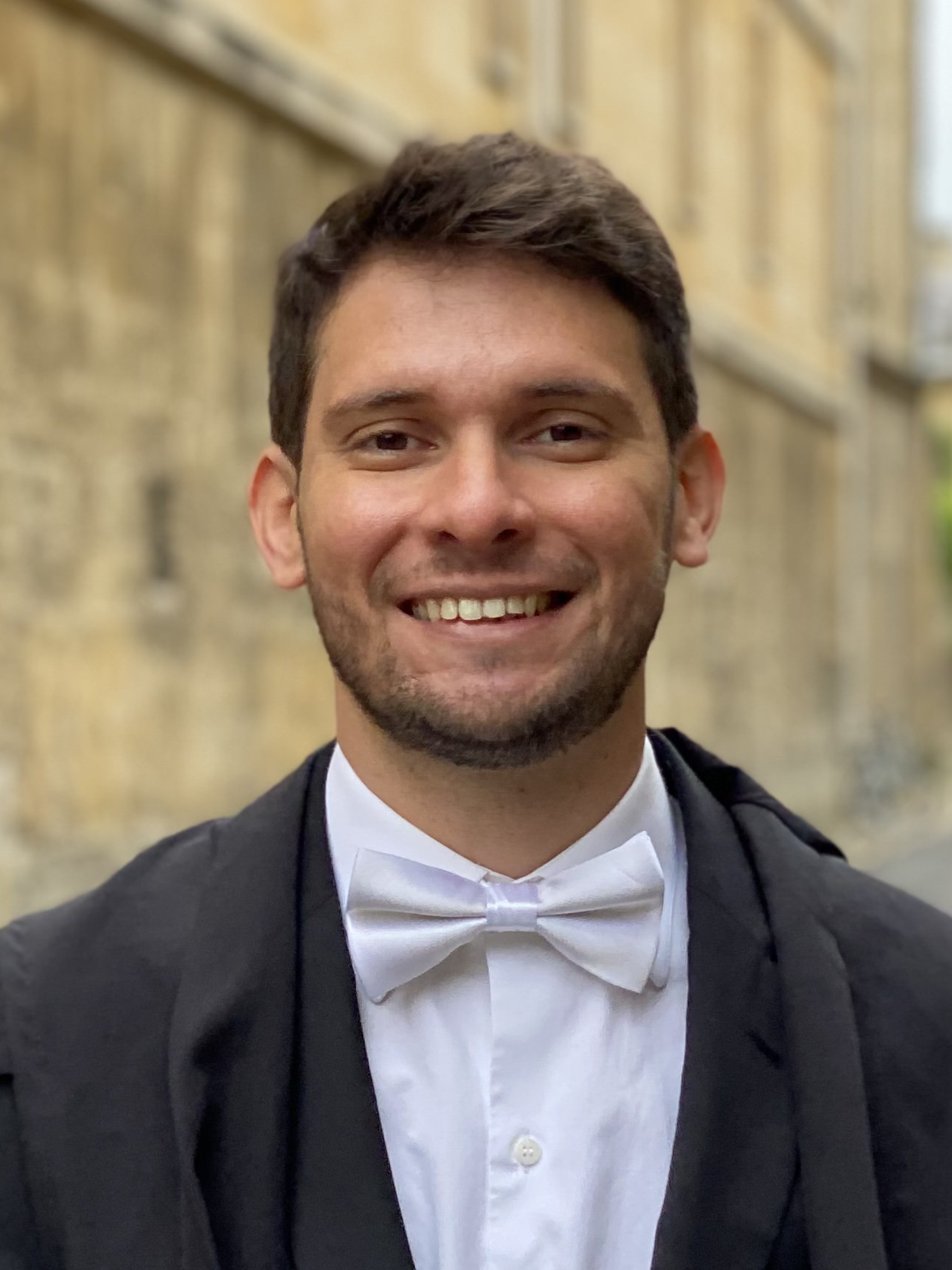}}]{Licio Romao}
	received the B.Eng. degree from the Federal University of Campina Grande (UFCG), Brazil, in 2014, the M.Eng. degree from the University of Campinas (UNICAMP), Brazil, in 2017, and the D.Phil. (Ph.D.) degree in Engineering Science from the University of Oxford, United Kingdom, in 2021. He is currently a postdoctoral research assistant at the Department of Computer Science, University of Oxford. His research interests include optimization and control strategies applied to large-scale, uncertain systems, as well as automatic verification and stochastic control with application to safety-critical systems. He is recipient of the 2021 IET Control and Automation Doctoral Dissertation prize.
\end{IEEEbiography}

\begin{IEEEbiography}[{\includegraphics[width=1in,height=1.25in,clip,keepaspectratio]{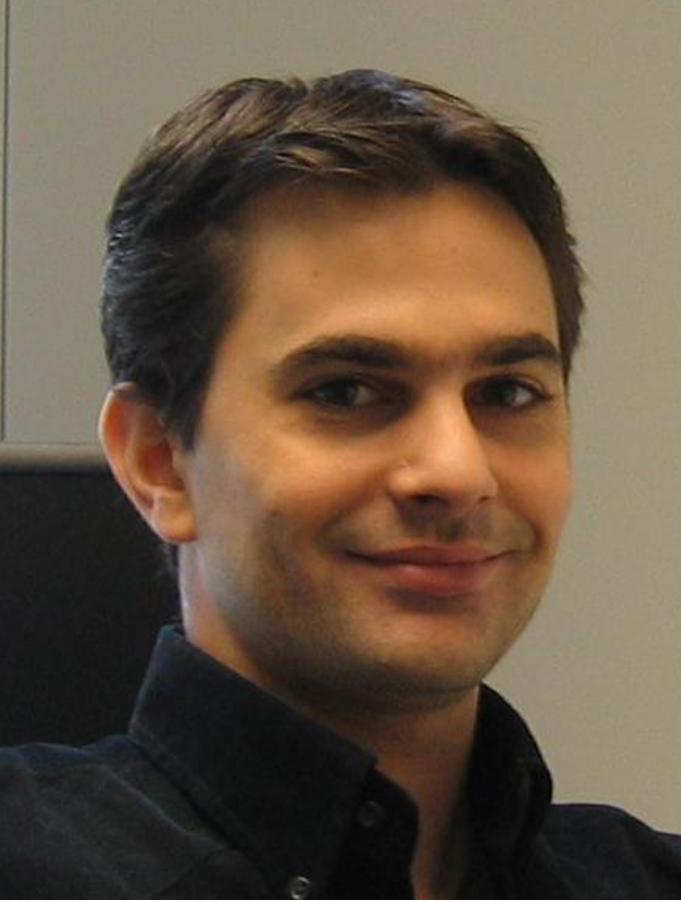}}]{Antonis Papachristodoulou}
	FIEEE received the M.A./M.Eng. degree in
	electrical and information sciences from the University of Cambridge,
	Cambridge, U.K., and the Ph.D. degree in control and dynamical systems
	(with a minor in aeronautics) from the California Institute of
	Technology, Pasadena, CA, USA. 	He is currently Professor of Engineering
	Science at the University of Oxford, Oxford, U.K., and a Tutorial Fellow
	at Worcester College, Oxford, as well as the Director of the EPSRC $\&$ BBSRC
	Centre for Doctoral Training in Synthetic Biology. He was previously an
	EPSRC Fellow. His research interests include
	large-scale nonlinear systems analysis, sum of squares programming,
	synthetic and systems biology, networked systems, and flow control.
	Professor Papachristodoulou received the 2015 European Control Award for
	his contributions to robustness analysis and applications to networked
	control systems and systems biology. In the same year, he received the
	O. Hugo Schuck Best Paper Award.
\end{IEEEbiography}

\begin{IEEEbiography}[{\includegraphics[width=1in,height=1.25in,clip,keepaspectratio]{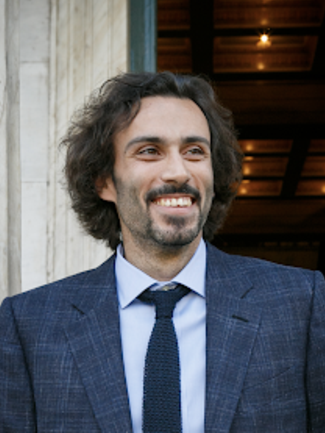}}]{Kostas Margellos} received the Diploma in electrical engineering from the University of Patras, Greece, in 2008, and the Ph.D. in control engineering from ETH Zurich, Switzerland, in 2012. He spent 2013, 2014 and 2015 as a postdoctoral researcher at ETH Zurich, UC Berkeley and Politecnico di Milano, respectively. In 2016 he joined the Control Group, Department of Engineering Science, University of Oxford, where he is currently an Associate Professor. He is also a Fellow of Reuben College and a Lecturer at Worcester College. His research interests include optimization and control of complex uncertain systems, with applications to energy and transportation networks.
\end{IEEEbiography}

\end{document}